\documentclass[11pt, a4paper]{article}
\usepackage[T1]{fontenc}
\usepackage{amsfonts}
\usepackage{amsmath}
\usepackage{amssymb}
\usepackage{amsthm}
\usepackage{bbm}
\usepackage{bm}
\usepackage{mathrsfs}
\usepackage{verbatim}
\usepackage{setspace}
\usepackage{color}
\usepackage{pdfsync}
\usepackage{enumitem}
\usepackage{cancel}

\theoremstyle{plain}
\newtheorem{theorem}{Theorem}[section]
\newtheorem{proposition}[theorem]{Proposition}
\newtheorem{lemma}[theorem]{Lemma}
\newtheorem{corollary}[theorem]{Corollary}

\theoremstyle{definition}
\newtheorem{definition}[theorem]{Definition}
\newtheorem{remark}[theorem]{Remark}

\newtheorem{assumption}{Assumption}
\newtheorem{assumptionBis}{Assumption}
\newtheorem{assumptionBar}{Assumption}
\theoremstyle{remark}

\renewcommand{\labelenumi}{(\roman{enumi})}

\renewenvironment{thebibliography}[1]{%
\begin{oldthebibliography}{#1}%
\setlength{\baselineskip}{.9em}
\linespread{.9}
\small
\setlength{\parskip}{0ex}%
\setlength{\itemsep}{.1em}%
}%
{%
\end{oldthebibliography}%
}
\newcommand{\eps}{\varepsilon}
\newcommand{\vp}{\varphi}
\newcommand{\M}{\mathbb{M}}
\newcommand{\N}{\mathbb{N}}

\newcommand{\R}{\mathbb{R}}

\newcommand{\F}{\mathbb{F}}

\newcommand{\cP}{\mathcal{P}}

\newcommand{\cR}{\mathcal{R}}
\newcommand{\cA}{\mathcal{A}}

\newcommand{\cD}{\mathcal{D}}

\newcommand{\cF}{\mathcal{F}}

\newcommand{\cH}{\mathcal{H}}

\newcommand{\cM}{\mathcal{M}}

\newcommand{\cO}{\mathcal{O}}

\newcommand{\cT}{\mathcal{T}}
\newcommand{\cU}{\mathcal{U}}

\newcommand{\bS}{\mathbf{S}}
\newcommand{\bD}{\mathbf{D}}

\DeclareMathOperator{\Int}{int}

\DeclareMathOperator{\tr}{Tr}

\newcommand{\x}{\times}
\newcommand{\as}{\mbox{-a.s.}}

\newcommand{\1}{\mathbf{1}}

\newcommand{\hS}{\hat{S}}
\newcommand{\hbS}{\hat{\bS}}

\newcommand{\Beqref}[1]{(B\ref{#1})}

\newcommand{\Bbiseqref}[1]{(B\ref{#1}')}

\numberwithin{equation}{section}

\begin{document}

\title{\vspace{-1.5em} Weak Dynamic Programming for Generalized State Constraints
\date{\today}%
\author{
  Bruno Bouchard%
  \thanks{
  CEREMADE, Universit\'e Paris Dauphine and CREST-ENSAE, \texttt{bouchard@ensae.fr}
  }
  \and
  Marcel Nutz%
  \thanks{
  Dept.\ of Mathematics, Columbia University, New York, \texttt{mnutz@math.columbia.edu}.
  Research supported by Swiss National Science Foundation Grant PDFM2-120424/1.
  The authors thank the two referees for their helpful comments.
  }
 }
}
\maketitle \vspace{-1em}

\begin{abstract}
We provide a dynamic programming principle for stochastic optimal control problems with expectation constraints.
A weak formulation, using test functions and a probabilistic relaxation of the constraint, avoids restrictions related to a measurable selection but still implies the Hamilton-Jacobi-Bellman equation in the viscosity sense. We treat open state constraints as a special case of expectation constraints and prove a comparison theorem to obtain the equation for closed state constraints.
\end{abstract}

\vspace{.5em}

{\small

}\\

\def\bruno#1{{\color{blue} #1}}

\section{Introduction}

We study the problem of stochastic optimal control under state constraints. In the most classical case, this is the problem of maximizing an expected reward, subject to the constraint that the controlled state process has to remain in a given subset $\cO$ of the state space.
There is a rich literature on the associated partial differential equations (PDEs), going back to \cite{Soner.86a, Soner.86b, IshiiKoike.96, Soravia.99a, Soravia.99b} in the first order case and \cite{LasryLions.89, Ka94, IsLo02} in the second order case.
The connection between the control problem and the equation is given by the dynamic programming principle. However, in the stochastic case, it is frequent practice to make this connection only formally, and in fact, we are not aware of a generally applicable, rigorous technique in the literature. Of course, there are specific situations where it is indeed possible to avoid proving the state-constrained dynamic programming principle; in particular, penalization arguments can be useful to reduce to the unconstrained case (e.g., \cite{Ka94}). We refer to \cite{FlemingSoner.06, YongZhou.99} for further background. %

Generally speaking, it is difficult to prove the dynamic programming principle when the regularity of the value function is not known a priori, due to certain measurable selection problems.
It was observed in~\cite{BT10} that, in the unconstrained case, these difficulties can be avoided by a weak formulation of the dynamic programming principle where the value function is replaced by a test function. This formulation, which is tailored to the derivation of the PDE in the sense of viscosity solutions, avoids the measurable selection and uses only a simple covering argument.
It turns out that the latter does not extend directly to the case with state constraints. Essentially, the reason is that if $\nu$ is some admissible control for the initial condition $x\in \cO$---i.e., the controlled state process $X^\nu_{t,x}$ started at $x$ remains in $\cO$---then $\nu$ may fail to have this property for a nearby initial condition $x'\in\cO$.

However, if $\cO$ is open and mild continuity assumptions are satisfied, then $X^\nu_{t,x'}$ will violate the state constraint with at most small probability when $x'$ is close to $x$.
This observation leads us to consider optimization problems with constraints in probability, and more generally expectation constraints of the form $E[g(X^\nu_{t,x}(T))]\leq m$ for given $m\in\R$. We shall see that, following the idea of~\cite{BET09}, such problems are amenable to dynamic programming if the constraint level $m$ is formulated dynamically via an auxiliary family of martingales. A key insight of the present paper is that relaxing the level $m$ by a small constant allows to prove a weak dynamic programming principle for general expectation constraint problems (Theorem~\ref{th:DPP}), while the PDE can be derived despite the relaxation.
We shall then obtain the dynamic programming principle for the classical state constraint problem (Theorem~\ref{thm: DPP state constraint}) by passing to a limit $m\downarrow 0$, with a suitable choice of $g$ and certain regularity assumptions. Of course, expectation constraints are of independent interest and use; e.g., in Mathematical Finance, where one considers the problem of maximizing expected terminal wealth $E[X_{0,x}^{\nu}(T)]$ under the loss constraint $E[([X_{0,x}^{\nu}(T)-x]^{-})^{p}]\leq m$, for some given $m,p>0$.

We exemplify the use of these results in the setting of controlled diffusions and show how the PDEs for expectation constraints and state constraints can be derived (Theorems~\ref{thm: PDE characterization V} and~\ref{thm: PDE state constraint}). For the latter case, we introduce an appropriate continuity condition at the boundary, under which we prove a comparison theorem. While the above concerned an open set $\cO$ and does not apply directly to the closed domain $\overline{\cO}$, we show via the comparison result that the value function for $\overline{\cO}$ coincides with the one for $\cO$, under certain conditions.

The remainder of the paper is organized as follows. In Section~\ref{se:DPPexpectationConstraints} we introduce an abstract setup for dynamic programming under expectation constraints and prove the corresponding relaxed weak dynamic programming principle. In Section~\ref{sec: DPP state constraint} we deduce the dynamic programming principle for the state constraint $\cO$. We specialize to the case of controlled diffusions in Section~\ref{se:ApplControlledDiff}, where we study the Hamilton-Jacobi-Bellman PDEs for expectation and state constraints. Appendix~\ref{se:appendixComparison} provides the comparison theorem.

Throughout this paper, (in)equalities between random variables are in the almost sure sense and relations between processes are up to evanescence, unless otherwise stated.

\section{Dynamic Programming Principle for Expectation Constraints}\label{se:DPPexpectationConstraints}

We fix a time horizon $T\in(0,\infty)$ and a probability space $(\Omega,\cF,P)$ equipped with a filtration
$\F=(\cF_t)_{t\in[0,T]}$. For each $t\in[0,T]$, we are given a set $\cU_t$ whose elements are seen as controls at time $t$. Given a separable metric space $S$, the \emph{state space}, we denote by
$\bS:=[0,T]\times S$ the time-augmented state space.
For each $(t,x)\in\bS$ and $\nu\in\cU_t$,
we are given a c\`adl\`ag adapted process $X_{t,x}^\nu=\{X_{t,x}^\nu(s),\,s\in[t,T]\}$ with values in $S$, the controlled \emph{state process}.
Finally, we are given two measurable functions $f,g: S\to \R$. We assume that
\begin{equation}\label{eq:controlsIntegrable}
  E[|f(X_{t,x}^\nu(T))|]<\infty \quad\mbox{and}\quad E[|g(X_{t,x}^\nu(T))|]<\infty\quad\mbox{for all }\nu\in\cU_t,
\end{equation}
so that the \emph{reward and constraint functions}
\[
  F(t,x;\nu):=E[f(X_{t,x}^\nu(T))],\quad G(t,x;\nu):=E[g(X_{t,x}^\nu(T))]%
\]
are well defined.
We also introduce the \emph{value function}
\begin{equation}\label{eq:valueFctDef}
  V(t,x,m):=\sup_{\nu\in\cU(t,x,m)} F(t,x;\nu),\quad  (t,x,m)\in\hbS,
\end{equation}
where $\hbS:=\bS\times\R\equiv [0,T]\times S\times \R$
and
\begin{equation}\label{eq:admControls}
  \cU(t,x,m):=\big\{\nu\in\cU_t:\, G(t,x;\nu)\leq m\big\}
\end{equation}
is the set of controls admissible at constraint level $m$. Here $\sup \emptyset:= -\infty$.

The following observation is the heart of our approach: Given a control $\nu$ admissible at level $m$ at the point $x$, relaxing the level $m$ will make $\nu$ admissible in an entire neighborhood of $x$. This will be crucial for the covering arguments used below. We use the acronym u.s.c.\ (l.s.c.) to indicate upper (lower) semicontinuity.

\begin{lemma}\label{le:nbhdForCovering}
  Let $(t,x,m)\in\hbS$, $\nu\in\cU(t,x,m)$ and assume that the function $x'\mapsto G(t,x';\nu)$ is u.s.c.\ at $x$. For each $\delta>0$ there exists a neighborhood $B$ of $x\in S$ such that $\nu\in \cU(t,x',m'+\delta)$ for all $x'\in B$ and all $m'\geq m$.
\end{lemma}

\begin{proof}
  We have $G(t,x;\nu)\leq m$ by the definition~\eqref{eq:admControls} of $\cU(t,x,m)$. In view of the upper semicontinuity, there exists a neighborhood $B$ of $x$ such that $G(t,x';\nu)\leq m+\delta \leq m'+ \delta$ for $x'\in B$; that is, $\nu\in \cU(t,x',m'+\delta)$.
\end{proof}

A control problem with an expectation constraint of the form~\eqref{eq:admControls} is not amenable to dynamic programming if we just consider a fixed level $m$. Extending an idea from~\cite{BET09, BEI10}, we shall see that this changes if the constraint is formulated dynamically by using auxiliary martingales.
To this end, we consider for each $t\in[0,T]$ a family $\cM_{t,0}$ of c\`adl\`ag martingales $M=\{M(s),\,s\in[t,T]\}$ with initial value $M(t)=0$. We also introduce
\[
  \cM_{t,m}:=\{m+M:\, M\in\cM_{t,0}\},\quad m\in\R.
\]
We assume that, for all $(t,x)\in\bS$ and $\nu\in\cU_t$,
\begin{equation}\label{eq:martRichness}
  \mbox{there exists $M_{t}^\nu[x]\in\cM_{t,m}$ such that $M_{t}^\nu[x](T)=g(X_{t,x}^\nu(T))$,}
\end{equation}
where, necessarily, $m=E[g(X_{t,x}^\nu(T))]$. In particular, given $\nu\in\cU_t$, the set
\[
  \cM_{t,m,x}^+(\nu):=\{M\in\cM_{t,m}:\, M(T)\geq g(X_{t,x}^\nu(T))\}
\]
is nonempty for any $m\geq E[g(X_{t,x}^\nu(T))]$. More precisely, we have the following characterization.

\begin{lemma}\label{le:reformulationOfConstraint}
  Let $(t,x)\in\bS$ and $m\in\R$. Then
  \[
   \cU(t,x,m)=\big\{\nu\in\cU_t:\, \cM_{t,m,x}^+(\nu)\neq\emptyset \big\}.
  \]
\end{lemma}

\begin{proof}
  Let $\nu\in\cU_t$. If there exists some $M\in \cM_{t,m,x}^+(\nu)$, then taking expectations yields $E[g(X_{t,x}^\nu(T))]\leq E[M(T)]=m$ and hence
  $\nu\in\cU(t,x,m)$.

  Conversely, let $\nu\in\cU(t,x,m)$; i.e., we have $m':=E[g(X_{t,x}^\nu(T))]\leq m$.
  With $M_{t}^\nu[x]$ as in~\eqref{eq:martRichness}, $M:=M_{t}^\nu[x]+m-m'$ is an element of $\cM_{t,m,x}^+(\nu)$.
\end{proof}

It will be useful in the following to consider for each $t\in[0,T]$ an auxiliary subfiltration $\F^t=(\cF^t_s)_{s\in[0,T]}$ of $\F$ such that $X^\nu_{t,x}$ and $M$ are $\F^t$-adapted for all $x\in S$, $\nu\in\cU_t$ and $M\in\cM_{t,0}$. Moreover, we denote by
$\cT^t$ the set of $\F^t$-stopping times with values in $[t,T]$.

The following assumption corresponds to one direction of the dynamic programming principle; cf.\ Theorem~\ref{th:DPP}(i) below.

\begin{assumption}\label{ass:A}
  For all $(t,x,m)\in\hbS$, $\nu\in\cU(t,x,m)$, $M\in\cM_{t,m,x}^+(\nu)$, $\tau\in\cT^t$ and $P$-a.e.\ $\omega\in\Omega$, there exists $\nu_\omega\in\cU(\tau(\omega),X_{t,x}^\nu(\tau)(\omega),M(\tau)(\omega))$ such that
  \begin{equation}\label{eq:assA}
    E\big[f(X_{t,x}^\nu(T))\big|\cF_\tau\big] (\omega) \leq F\big( \tau(\omega), X_{t,x}^\nu(\tau)(\omega);\nu_\omega \big).
  \end{equation}
\end{assumption}

Next, we state two variants of the assumptions for the converse direction of the dynamic programming principle; we shall comment on the differences in Remark~\ref{rk:versionsOfDPP} below. In the first variant, the intermediate time is deterministic; this will be enough to cover stopping times with countably many values in Theorem~\ref{th:DPP}(ii) below.  We recall the notation $M_{t}^\nu[x]$ from~\eqref{eq:martRichness}.

\begin{assumption}\label{ass:B}
  Let $(t,x)\in\bS$, $\nu\in\cU_t$, $s\in[t,T]$, $\bar{\nu}\in \cU_s$ and $\Gamma\in\cF^t_s$.

  \begin{enumerate}[topsep=3pt, partopsep=0pt, itemsep=1pt,parsep=2pt]
  \renewcommand{\labelenumi}{({\bf B\arabic{enumi}})}

   \item \label{ass:B1} There exists a control $\tilde{\nu} \in \cU_t$ such that
    \begin{align}
      X_{t,x}^{\tilde{\nu}}(\cdot)&=X_{t,x}^\nu(\cdot)&&\mbox{on }[t,T]\times (\Omega\setminus \Gamma);\label{eq:causalityFIN}\\
      X_{t,x}^{\tilde{\nu}}(\cdot)&=X_{s,X_{t,x}^\nu(s)}^{\bar{\nu}} (\cdot) &&\mbox{on }[s,T]\times \Gamma;\label{eq:causalitySwitchFIN} \\
      E\big[ f(X_{t,x}^{\tilde{\nu}}(T)) \big|\cF_s\big] &\geq F(s,X_{t,x}^\nu(s);\bar{\nu}) &&\mbox{on } \Gamma.\label{eq:consistencyFIN}
    \end{align}
    The control $\tilde{\nu}$ is denoted by $\nu\otimes_{(s,\Gamma)}\bar{\nu}$ and called a \emph{concatenation} of $\nu$ and $\bar{\nu}$ on $(s,\Gamma)$.

   \item \label{ass:B2} Let $M\in\cM_{t,0}$. There exists a process $\bar{M}=\{\bar{M}(r),\,r\in [s,T]\}$ such that
    \[
      \bar{M}(\cdot)(\omega) = \big(M_{s}^{\bar{\nu}}[X_{t,x}^\nu(s)(\omega)] (\cdot)\big)(\omega)\quad \mbox{on } [s,T]\quad P\as
    \]
    and
    \[
      M\1_{[t,s)} + \1_{[s,T]}\Big( M\1_{\Omega\setminus\Gamma} + \big[\bar{M} - \bar{M}(s)+M(s)\big]\1_\Gamma \Big) \in\cM_{t,0}.
    \]

   \item  \label{ass:B3} Let $m\in\R$ and $M\in\cM_{t,m,x} ^+(\nu)$. For $P$-a.e.\ $\omega\in\Omega$,
    there exist a control $\nu_\omega\in\cU(s,X_{t,x}^\nu(s)(\omega),M(s)(\omega))$.
  \end{enumerate}
\end{assumption}

In the second variant, the intermediate time is a stopping time and we have an additional assumption about the structure of the sets $\cU_s$. This variant corresponds to Theorem~\ref{th:DPP}(ii') below.

\addtocounter{assumptionBis}{1}
\begin{assumptionBis}\label{ass:B'}
  Let $(t,x)\in\bS$, $\nu\in\cU_t$, $\tau\in\cT^t$, $\Gamma\in\cF^t_\tau$ and $\bar{\nu}\in \cU_{\|\tau\|_{L^\infty}}$.
  \begin{enumerate}[topsep=5pt, partopsep=0pt, itemsep=1pt,parsep=2pt]
  \renewcommand{\labelenumi}{({\bf B\arabic{enumi}'})}
  \addtocounter{enumi}{-1}

   \item \label{ass:B0'} $\cU_{s}\supseteq \cU_{s'}$ for all $0\leq s\leq s'\leq T$.

   \item \label{ass:B1'} There exists a control $\tilde{\nu} \in \cU_t$, denoted by $\nu\otimes_{(\tau,\Gamma)}\bar{\nu}$, such that
    \begin{align}
      X_{t,x}^{\tilde{\nu}}(\cdot)&=X_{t,x}^\nu(\cdot)&&\mbox{on }[t,T]\times (\Omega\setminus \Gamma);\nonumber\\
      X_{t,x}^{\tilde{\nu}}(\cdot)&=X_{\tau,X_{t,x}^\nu(\tau)}^{\bar{\nu}} (\cdot) &&\mbox{on }[\tau,T]\times \Gamma;\label{eq:causalitySwitch} \\
      E\big[ f(X_{t,x}^{\tilde{\nu}}(T)) \big|\cF_\tau\big] &\geq F(\tau,X_{t,x}^\nu(\tau);\bar{\nu}) &&\mbox{on } \Gamma.\label{eq:consistency}
    \end{align}

   \item \label{ass:B2'} Let $M\in\cM_{t,0}$. There exists a process $\bar{M}=\{\bar{M}(r),\,r\in [\tau,T]\}$ such that
    \[
      \bar{M}(\cdot)(\omega) = \big(M_{\tau(\omega)}^{\bar{\nu}}[X_{t,x}^\nu(\tau)(\omega)] (\cdot)\big)(\omega)\quad \mbox{on } [\tau,T]\quad P\as
    \]
    and
    \[
      M\1_{[t,\tau)} + \1_{[\tau,T]}\Big( M \1_{\Omega\setminus\Gamma} + \big[\bar{M} - \bar{M}(\tau)+M(\tau)\big]\1_{\Gamma}\Big) \in\cM_{t,0}.
    \]

   \item  \label{ass:B3'} Let $m\in\R$ and $M\in\cM_{t,m,x} ^+(\nu)$. For $P$-a.e.\ $\omega\in\Omega$,
    there exist a control $\nu_\omega\in\cU(\tau(\omega),X_{t,x}^\nu(\tau)(\omega),M(\tau)(\omega))$.
  \end{enumerate}
\end{assumptionBis}

\begin{remark}\label{rem: ass A implies B3}{\rm
  \begin{enumerate}[topsep=3pt, partopsep=0pt, itemsep=1pt,parsep=2pt]

    \item Assumption~\ref{ass:B'} implies Assumption~\ref{ass:B}. Moreover, Assumption~\ref{ass:A} implies~\Bbiseqref{ass:B3'}.

    \item Let $\bD:=\{(t,x,m)\in\hbS: \cU(t,x,m)\neq\emptyset\}$ denote the natural domain of our optimization problem.
     Then~\Bbiseqref{ass:B3'} can be stated as follows: for any $\nu\in \cU(t,x,m)$ and $M\in\cM^+_{t,m,x}(\nu)$, \begin{equation}\label{eq:domainInvariance}
       \big(\tau,X_{t,x}^\nu(\tau),M(\tau)\big)\in\bD\quad P\as\quad\mbox{for all }\tau\in\cT^t.
     \end{equation}
     Under~\Beqref{ass:B3}, the same holds if $\tau$ takes countably many values.
     We remark that the invariance property~\eqref{eq:domainInvariance} corresponds to one direction in the geometric dynamic programming of~\cite{SonerTouzi.02b}. %

    \item We note that~\eqref{eq:causalitySwitch} (and similarly~\eqref{eq:causalitySwitchFIN}) states in particular that
    \[
      (r,\omega)\mapsto \big(X_{\tau,X_{t,x}^\nu(\tau)}^{\bar{\nu}} (r)\big)(\omega)
      := \big(X_{\tau(\omega),X_{t,x}^\nu(\tau)(\omega)}^{\bar{\nu}} (r)\big)(\omega)
    \]
    is a well-defined adapted process (up to evanescence) on $[t,T]\times \Gamma$. Of course, this is an implicit measurability condition on $(t,x)\mapsto X_{t,x}^{\bar{\nu}}$.

    \item For an illustration of~\Bbiseqref{ass:B1'}, let us assume that the controls are
    predictable processes. In this case, one can often take
    \begin{equation}\label{eq:concatenationExample}
      \nu\otimes_{(\tau,\Gamma)}\bar{\nu}:=\nu\1_{[0,\tau]} + \1_{(\tau,T]} \big(\bar{\nu}\1_{\Gamma} + \nu\1_{\Omega\setminus \Gamma}\big)
    \end{equation}
    and the condition that $\nu\otimes_{(\tau,\Gamma)}\bar{\nu}\in\cU_t$ is called \emph{stability under concatenation}. The idea is that we use the control $\nu$ up to time $\tau$; after time $\tau$, we continue using $\nu$ in the event $\Omega\setminus\Gamma$ while we switch to the control $\bar{\nu}$ in the event $\Gamma$.
    Of course, one can omit the set $\Gamma\in\cF^t_\tau$ by observing that
    \[
      \nu\otimes_{(\tau,\Gamma)}\bar{\nu} = \nu\1_{[0,\tau']} + \bar{\nu}\1_{(\tau',T]}
    \]
    for the $\F^t$-stopping time $\tau':=\tau\1_\Gamma+T\1_{\Omega\setminus\Gamma}$.

  \end{enumerate}}
\end{remark}

We can now state our weak dynamic programming principle for the stochastic control problem~\eqref{eq:valueFctDef} with expectation constraint. The formulation is weak in two ways; namely, the value function is replaced by a test function and the constraint level $m$ is relaxed by an arbitrarily small constant $\delta>0$ (cf.\ \eqref{eq:DPPii} below).
The flexibility of choosing the set $\cD$ appearing below, will be used in Section~\ref{sec: DPP state constraint}. We recall the set $\bD$ introduced in Remark~\ref{rem: ass A implies B3}(ii).

\begin{theorem}\label{th:DPP}
  Let $(t,x,m)\in\hbS$, $\nu\in\cU(t,x,m)$ and $M\in \cM^+_{t,m,x}(\nu)$. Let $\tau\in\cT^t$ and let
  $\cD\subseteq \hbS$ be a set such that $(\tau,X_{t,x}^\nu(\tau),M(\tau))\in \cD$ holds $P$-a.s.

  \begin{enumerate}[topsep=3pt, partopsep=0pt, itemsep=1pt,parsep=2pt]
   \item Let Assumption~\ref{ass:A} hold true and let $\varphi: \hbS\to[-\infty,\infty]$ be a measurable function such that $V\leq \varphi$ on $\cD$. Then $E \big[ \varphi(\tau,X_{t,x}^\nu(\tau),M(\tau))^- \big]<\infty$ and
      \begin{equation}\label{eq:DPPi}
        F(t,x;\nu) \leq E \big[ \varphi(\tau,X_{t,x}^\nu(\tau),M(\tau)) \big].
      \end{equation}

  \item Let $\delta>0$, let Assumption~\ref{ass:B} hold true and assume that $\tau$ takes countably many values
  $(t_i)_{i\geq 1}$. Let $\varphi: \hbS\to[-\infty,\infty)$ be a measurable function such that $V\geq \varphi$ on $\cD$.
  Assume that for any fixed $\bar{\nu}\in\cU_{t_i}$,
  \begin{equation}\label{eq:semicontAssFIN}
    \left.
    \begin{array}{rcll}
      (x',m')&\!\!\!\mapsto\!\!\!& \varphi(t_i,x',m')&\mbox{is u.s.c.} \\[.2em]
      x'&\!\!\!\mapsto\!\!\!& F(t_i,x';\bar{\nu})&\mbox{is l.s.c.} \\[.2em]
      x'&\!\!\!\mapsto\!\!\!& G(t_i,x';\bar{\nu})&\mbox{is u.s.c.}
    \end{array}
    \right\}\mbox{ on }\cD^i
  \end{equation}
  for all $i\geq 1$, where $\cD^i:=\{(x',m'):\, (t_i,x',m')\in\cD\}\subseteq S\times\R$. Then
    \begin{equation}\label{eq:DPPii}
      V(t,x,m+\delta) \geq E \big[ \varphi(\tau,X_{t,x}^\nu(\tau),M(\tau)) \big].
    \end{equation}
  \item[(ii')] Let $\delta>0$ and let Assumption~\ref{ass:B'} hold true. Let $\varphi: \hbS\to[-\infty,\infty)$ be a measurable function such that $V\geq \varphi$ on $\cD$.
  Assume that $\cD\cap\bD\subseteq \hbS$ is $\sigma$-compact and that for any fixed $\bar{\nu}\in\cU_{t_0}$, $t_0\in[t,T]$,
  \begin{equation}\label{eq:semicontAss}
    \left.
    \begin{array}{rcll}
      (t',x',m')&\!\!\!\mapsto\!\!\!& \varphi(t',x',m')&\mbox{is u.s.c.} \\[.2em]
      (t',x')&\!\!\!\mapsto\!\!\!& F(t',x';\bar{\nu})&\mbox{is l.s.c.} \\[.2em]
      (t',x')&\!\!\!\mapsto\!\!\!& G(t',x';\bar{\nu})&\mbox{is u.s.c.}
    \end{array}
    \right\}\mbox{ on }\cD \cap \{t'\leq t_0\}.
  \end{equation}
  Then~\eqref{eq:DPPii} holds true.
  \end{enumerate}
\end{theorem}

The following convention is used on the right hand side of~\eqref{eq:DPPii}: if $Y$ is any random variable,
\begin{equation}\label{eq:expConvention}
  \mbox{we set }E[Y]:=-\infty\quad\mbox{if}\quad E[Y^+]=E[Y^-]=\infty.
\end{equation}
We note that Assumption~\Bbiseqref{ass:B0'} ensures that the expressions $F(t',x';\bar{\nu})$ and $G(t',x';\bar{\nu})$ in \eqref{eq:semicontAss} are well defined for $t'\leq t_0$.

\begin{remark}\label{rk:versionsOfDPP}{\rm
 The difference between parts (ii) and (ii') of the theorem stems from the fact that in the proof of (ii') we consider $[0,T]\times S\times\R$ as the state space while for (ii) it suffices to consider $S\times\R$ and hence no assumptions on the time variable are necessary. Regarding applications, (ii) is obviously the better choice if stopping times with countably many values (and in particular deterministic times) are sufficient.

 There is a number of cases where the extension to a general stopping time $\tau$ can be accomplished a posteriori by approximation, in particular if one has a priori estimates for the value function so that one can restrict to test functions with specific growth properties. Assume for illustration that $f$ is bounded from above, then so is $V$ and one will be interested only in test functions $\varphi$ which are bounded from above; moreover, it will typically not hurt to assume that $\varphi$ is u.s.c.\ (or even continuous) in all variables. Now let $(\tau_n)$ be a sequence of stopping time taking finitely many (e.g., dyadic) values such that $\tau_n\downarrow \tau$ $P$-a.s. Applying~\eqref{eq:DPPii} to each $\tau^n$ and using Fatou's lemma as well as the right-continuity of the paths of $X^\nu_{t,x}$ and $M$, we then find that
 \eqref{eq:DPPii} also holds for the general stopping time $\tau$.

 On the other hand, it is not always possible to pass to the limit as above and then (ii') is necessary to treat general stopping times. The compactness assumption should be reasonable \emph{provided that} $S$ is $\sigma$-compact; e.g., $S=\R^d$. Then, the typical way to apply (ii') is to take $(t,x,m)\in\Int \bD$ and let $\cD$ be a open or closed neighborhood of $(t,x,m)$ such that $\cD\subseteq \bD$.}
\end{remark}

\begin{proof}[Proof of Theorem~\ref{th:DPP}]
  (i)~Fix $(t,x,m)\in\hbS$.
  With $\nu_\omega$ as in~\eqref{eq:assA}, the definition~\eqref{eq:valueFctDef} of $V$ and $V\leq \varphi$ on $\cD$ imply that
  \begin{align*}
    E\big[f(X_{t,x}^\nu(T))|\cF_\tau\big] (\omega)
    & \leq F\big( \tau(\omega), X_{t,x}^\nu(\tau)(\omega);\nu_\omega \big) \\
    & \leq V\big( \tau(\omega), X_{t,x}^\nu(\tau)(\omega),M(\tau)(\omega)\big) \\
    & \leq \varphi\big(\tau(\omega), X_{t,x}^\nu(\tau)(\omega),M(\tau)(\omega)\big)\quad P\as
  \end{align*}
  After noting that the left hand side is integrable by~\eqref{eq:controlsIntegrable}, the result follows by taking expectations.\\

  (ii)~Fix $(t,x,m)\in\hbS$ and let $\eps>0$.

  \emph{1. Countable Cover.}
  Like $\cD$, the set $\cD\cap\bD$ has the property that
  \begin{equation}\label{eq:cDinvariance}
    \big(\tau,X_{t,x}^\nu(\tau),M(\tau)\big)\in \cD\cap\bD\quad P\as;
  \end{equation}
  this follows from~\eqref{eq:domainInvariance} since $\tau$ takes countably many values.
  Replacing $\cD$ by $\cD\cap\bD$ if necessary, we may therefore assume that $\cD\subseteq \bD$.
  Using also that $V\geq \varphi$ on $\cD$, the definition~\eqref{eq:valueFctDef} of $V$ shows that there exists for each $(s,y,n)\in\cD$ some $\nu^{s,y,n}\in\cU(s,y,n)$ satisfying
  \begin{equation}\label{eq:epsOptimizers}
    F(s,y;\nu^{s,y,n})\geq \varphi(s,y,n)-\eps.
  \end{equation}
  Fix one of the points $t_i$ in time. For each $(y,n)\in \cD^i$, the semicontinuity assumptions~\eqref{eq:semicontAssFIN} and Lemma~\ref{le:nbhdForCovering} imply that there exists
  a neighborhood $(y,n)\in B^i(y,n)\subseteq \hat{S}:=S\times\R$ (of size depending on $y,n,i,\eps,\delta$)
  such that
  \begin{equation}\label{eq:coverControls}
    \left.
    \begin{array}{rcl}
      \nu^{t_i,y,n} & \!\!\!\in\!\!\! & \cU(t_i,y',n'+\delta) \\[.2em]
      \varphi(t_i,y',n') & \!\!\!\leq\!\!\! & \varphi(t_i,y,n) +\eps \\[.2em]
      F(t_i,y';\nu^{t_i,y,n}) & \!\!\!\geq\!\!\! & F(t_i,y;\nu^{t_i,y,n})- \eps\!\!
    \end{array}
    \right\}\mbox{ for all }(y',n')\in B^i(y,n)\cap \cD^i.
  \end{equation}
  Here the first inequality may read $-\infty \leq -\infty$.
  We note that $\cD^i\subseteq\hS$ is metric separable for the subspace topology relative to the product topology on $\hS$.
  Therefore, since the family $\{B^i(y,n)\cap\cD^i:\,(y,n)\in \hS\}$ forms an open cover of $\cD^i$,
  there exists a sequence $(y_j,n_j)_{j\geq1}$ in $\hS$ such that $\{B^i(y_j,n_j)\cap\cD^i\}_{j\geq1}$ is a countable subcover of $\cD^i$.
  We set $\nu^i_j:=\nu^{t_i,y_j,n_j}$ and $B^i_j:=B^i(y_j,n_j)$, so that
  \begin{equation}\label{eq:coverProperty}
    \cD^i\subseteq \bigcup_{j\geq1} B^i_j.
  \end{equation}
  We can now define, for $i$ still being fixed, a measurable partition $(A^i_j)_{j\geq1}$ of $\cup_{j\geq1} B^i_j$ by
  \[
    A^i_1:=B^i_1,\quad A^i_{j+1}:=B^i_{j+1}\setminus (B^i_1 \cup\dots\cup B^i_j),\quad j\geq1.
  \]
  Since $A^i_j\subseteq B^i_j$, the inequalities~\eqref{eq:epsOptimizers} and \eqref{eq:coverControls} yield that
  \begin{equation}\label{eq:ProofDppThreeEps}
    F(t_i,y';\nu^i_j)\geq \varphi(t_i,y',n') - 3\eps\quad\mbox{for all }(y',n')\in A^i_j\cap \cD^i.
  \end{equation}

  \emph{2. Concatenation.} Fix an integer $k\geq 1$; we now focus on $(t_i)_{1\leq i\leq k}$. We may assume that
  $t_1< t_2< \dots < t_k$, by eliminating and relabeling some of the $t_i$. We define the $\cF^t_\tau$-measurable sets
  \[
    \Gamma^i_j:=\big\{\tau= t_i \mbox{ and } (X_{t,x}^\nu(t_i),M(t_i))\in A^i_j\big\}\in\cF_{t_i} \quad\mbox{and}\quad
    \Gamma(k):=\bigcup_{1\leq i,j\leq k} \Gamma^i_j.
  \]
  Since the $t_i$ are distinct and $A^i_j\cap A^i_{j'}=\emptyset$ for $j\neq j'$, we have
  $\Gamma^i_j\cap \Gamma^{i'}_{j'}=\emptyset$ for $(i,j)\neq (i',j')$.
  We can then consider the successive concatenation
  \[
    \nu(k):=\nu\otimes_{(t_1,\Gamma^1_1)} \nu^1_1\otimes_{(t_1,\Gamma^1_2)}\nu^1_2\cdots\otimes_{(t_1,\Gamma^1_k)} \nu^1_k\otimes_{(t_2,\Gamma^2_1)} \nu^2_1 \cdots\otimes_{(t_k,\Gamma^k_k)}\nu^k_k,
  \]
  which is to be read from the left with $\nu^a\otimes\nu^b\otimes\nu^c:= (\nu^a\otimes\nu^b)\otimes\nu^c$.
  It follows from an iterated application of~\Beqref{ass:B1} that $\nu(k)$ is well defined and in particular
  $\nu(k)\in\cU_t$. (To understand the meaning of $\nu(k)$, it may be useful to note that in the example considered in~\eqref{eq:concatenationExample}, we would have
  \[
    \nu(k) = \nu\1_{[0,\tau]} + \1_{(\tau,T]} \bigg(\nu\1_{\Omega\setminus\Gamma(k)}
      + \sum_{1\leq i,j\leq k}\nu^i_j\1_{\Gamma^i_j} \bigg);
  \]
  i.e., at time $\tau$, we switch to $\nu^i_j$ on $\Gamma^i_j$.)
  We note that~\eqref{eq:causalityFIN} implies that $X_{t,x}^{\nu(k)}=X_{t,x}^{\nu}$ on $\Omega\setminus \Gamma(k)\in\cF_\tau$ and hence
  \begin{equation}\label{eq:ProofDppGammaComplement}
    E\big[ f(X_{t,x}^{\nu(k)}(T)) \big|\cF_\tau\big]=E\big[ f(X_{t,x}^{\nu}(T)) \big|\cF_\tau\big]\quad\mbox{on } \Omega\setminus \Gamma(k).
  \end{equation}
  Moreover, the fact that $\tau=t_i$ on $\Gamma^i_j$, repeated application of~\eqref{eq:causalityFIN}, and~\eqref{eq:consistencyFIN} show that
  \[
    E\big[ f(X_{t,x}^{\nu(k)}(T)) \big|\cF_\tau\big] \geq F(\tau,X_{t,x}^\nu(\tau);\nu^i_j)\quad\mbox{on }\Gamma^i_j,\quad\mbox{for } 1\leq i,j\leq k,
  \]
  and we deduce via~\eqref{eq:ProofDppThreeEps} that
  \begin{align}\label{eq:ProofDppGamma}
    E\big[ f(X_{t,x}^{\nu(k)}(T)) \big|\cF_\tau\big] \1_{\Gamma(k)}
     & \geq \sum_{1\leq i,j\leq k} F(\tau, X_{t,x}^\nu(\tau); \nu^i_j)\1_{\Gamma^i_j} \nonumber\\
     & \geq \sum_{1\leq i,j\leq k} \big(\varphi(\tau, X_{t,x}^\nu(\tau), M(\tau))-3\eps\big) \1_{\Gamma^i_j} \nonumber\\
     & \geq \varphi(\tau, X_{t,x}^\nu(\tau), M(\tau)) \1_{\Gamma(k)} - 3\eps.
  \end{align}

  \emph{3. Admissibility.} Next, we show that $\nu(k)\in\cU(t,x,m+\delta)$.
  By~\Beqref{ass:B2} there exists, for each pair $1\leq i,j\leq k$, a process $M^i_j=\{M^i_j(u),\, u\in[t_i,T]\}$ such that
  \begin{equation}\label{eq:defMij}
    M^i_j(\cdot)(\omega)%
    =\Big(M_{t_i}^{\nu^i_j}[X_{t,x}^\nu(t_i)(\omega)](\cdot)\Big)(\omega)\mbox{ on }\Gamma^i_j
  \end{equation}
  and such that
  \begin{align*}
   M^{(k)}&:= (M+\delta)\1_{[t,\tau)}\\
   &\phantom{:=}  + \1_{[\tau,T]} \bigg((M+\delta)\1_{\Omega\setminus\Gamma(k)}
      + \sum_{1\leq i,j\leq k}\big[M^i_j -M^i_j(t_i) + M(t_i)+\delta\big]\1_{\Gamma^i_j} \bigg)
  \end{align*}
  is an element of $\cM_{t,m+\delta}$.   We note that $M^{(k)}(T)\geq M^i_j(T)$ on $\Gamma^i_j$ since $M^i_j(t_i)\leq M(t_i)+\delta$ on $\Gamma^i_j$ as a result of the first condition in~\eqref{eq:coverControls}. Hence, using~\eqref{eq:causalitySwitchFIN} and~\eqref{eq:defMij}, we have
  \begin{equation}\label{eq:proofDPPMartOnGamma}
    g(X_{t,x}^{\nu(k)}(T))=g\big(X_{t_i,X_{t,x}^\nu(t_i)}^{\nu^i_j}(T)\big)\leq M^i_j(T)\leq M^{(k)}(T)\quad\mbox{on }\Gamma^i_j.
  \end{equation}
  This holds for all $1\leq i,j\leq k$. %
  On the other hand, using~\eqref{eq:causalityFIN} and that $M\in\cM^+_{t,m,x}(\nu)$, we have that
  \begin{equation}\label{eq:proofDPPuseAdmissibility}
    g(X_{t,x}^{\nu(k)}(T))=g(X_{t,x}^{\nu}(T))\leq M(T)\leq M(T)+\delta =M^{(k)}(T)\quad\mbox{on }\Omega\setminus\Gamma(k).
  \end{equation}
  Combining~\eqref{eq:proofDPPMartOnGamma} and~\eqref{eq:proofDPPuseAdmissibility}, we have $g(X_{t,x}^{\nu(k)}(T))\leq M^{(k)}(T)$ on $\Omega$ and so Lemma~\ref{le:reformulationOfConstraint} yields that $\nu(k)\in\cU(t,x,m+\delta)$.

  \emph{4. $\eps$-Optimality.} We may assume that
  either the positive or the negative part of $\varphi(\tau,X_{t,x}^\nu(\tau),M(\tau))$ is integrable, as otherwise our claim~\eqref{eq:DPPii} is trivial by~\eqref{eq:expConvention}. Using the definition~\eqref{eq:valueFctDef} of $V$ as well as~\eqref{eq:ProofDppGammaComplement} and~\eqref{eq:ProofDppGamma}, we have that
  \begin{align*}%
   V(t,x,m+\delta)
     &\geq E\big[f(X_{t,x}^{\nu(k)}(T))\big] \\
     & = E\big[E\big[ f(X_{t,x}^{\nu(k)}(T)) \big|\cF_\tau\big]\big] \\
     & \geq E\big[\varphi(\tau,X_{t,x}^\nu(\tau),M(\tau))\1_{\Gamma(k)}\big] - 3\eps
        + E\big[f(X_{t,x}^{\nu}(T)) \1_{\Omega\setminus \Gamma(k)}\big]
  \end{align*}
  for every $k\geq1$. Letting $k\to\infty$, we have $\Gamma(k)\uparrow \Omega$ $P$-a.s.\ by~\eqref{eq:cDinvariance} and~\eqref{eq:coverProperty}; therefore,
  \[
    E\big[f(X_{t,x}^{\nu}(T)) \1_{\Omega\setminus \Gamma(k)}\big]\to0
  \]
  by dominated convergence and~\eqref{eq:controlsIntegrable}.
  Moreover, monotone convergence yields
  \[
    E\big[\varphi(\tau,X_{t,x}^\nu(\tau),M(\tau))\1_{\Gamma(k)}\big] \to E\big[\varphi(\tau,X_{t,x}^\nu(\tau),M(\tau))\big];
  \]
  to see this, consider separately the cases where the positive or the negative part of
  $\varphi(\tau, X_{t,x}^\nu(\tau), M(\tau))$ are integrable.
  Hence we have shown that
  \[
    V(t,x,m+\delta) \geq E[\varphi(\tau,X_{t,x}^\nu(\tau),M(\tau))] - 3\eps.
  \]
  As $\eps>0$ was arbitrary, this completes the proof of (ii).\\

  (ii')~Fix $(t,x,m)\in\hbS$ and let $\eps>0$. In contrast to the proof of~(ii), we shall cover a subset of $\hbS$ rather than $S\times\R$. By~\eqref{eq:domainInvariance}, we may again assume that $\cD\cap\bD=\cD$.
  Since $V\geq \varphi$ on $\cD$, the definition~\eqref{eq:valueFctDef} of $V$ shows that there exists for each $(s,y,n)\in\cD$ some $\nu^{s,y,n}\in\cU(s,y,n)$ satisfying
  \begin{equation}\label{eq:epsOptimizers'}
    F(s,y;\nu^{s,y,n})\geq \varphi(s,y,n)-\eps.
  \end{equation}
  Given $(s,y,n)\in \cD$, the semicontinuity assumptions~\eqref{eq:semicontAss} and a variant of Lemma~\ref{le:nbhdForCovering} (including the time variable) imply that there exists a set $B(s,y,n)\subseteq \hbS$ of the form
  \[
    B(s,y,n)= \big((s-r,s]\cap [0,T]\big)\times B_r(y,n),
  \]
  where $r>0$ and $B_r(y,n)$ is an open ball in $S\times\R_+$,
  such that
  \begin{equation}\label{eq:coverControls'}
    \left.
    \begin{array}{rcl}
      \nu^{s,y,n} & \!\!\!\in\!\!\! & \cU(s',y',n'+\delta) \\[.2em]
      \varphi(s',y',n') & \!\!\!\leq\!\!\! & \varphi(s,y,n) +\eps \\[.2em]
      F(s',y';\nu^{s,y,n}) & \!\!\!\geq\!\!\! & F(s,y;\nu^{s,y,n})- \eps\!\!
    \end{array}
    \right\}\;\mbox{for all }(s',y',n')\in B(s,y,n)\cap \cD.
  \end{equation}
  Note that we have exploited Assumption~\Bbiseqref{ass:B0'}, which forces us to use the half-closed interval for the time variable. As a result,
  $B(s,y,n)$ is open for the product topology on $[0,T]\times S\times\R$ if $S$ and $\R$ are given the usual topology and
  $[0,T]$ is given the topology generated by the half-closed intervals (the upper limit topology). Let us denote the latter topological space by $[0,T]^*$.

  Like the Sorgenfrey line, $[0,T]^*$ is a Lindel\"of space.
  Let $\cD'$ be the canonical projection of $\cD\subseteq \hbS\equiv [0,T]\times S\times\R$ to $S\times\R$. Then $\cD'$ is again $\sigma$-compact. It is easy to see that the product of a Lindel\"of space with a $\sigma$-compact space is again Lindel\"of; in particular, $[0,T]^*\times \cD'$ is Lindel\"of. Moreover, since $\cD$ was $\sigma$-compact for the original topology on $\hbS$ and the topology of $[0,T]^*$ is finer than the original one, we still have that $\cD\subseteq[0,T]^*\times \cD'$ is a countable union of closed subsets and therefore $\cD$ is Lindel\"of also for the new topology.

  By the Lindel\"of property, the cover $\{B(s,y,n)\cap\cD:\,(s,y,n)\in \hbS\}$ of $\cD$ admits a countable subcover $\{B(s_j,y_j,n_j)\cap\cD\}_{j\geq1}$.
  We set $\nu_j:=\nu^{s_j,y_j,n_j}$ and $B_j:=B(s_j,y_j,n_j)$, then $\cD\subseteq \cup_{j\geq1} B_j$
  and
  \[
    A_1:=B_1,\quad A_{j+1}:=B_{j+1}\setminus (B_1 \cup\dots\cup B_j),\quad j\geq1
  \]
  defines a measurable partition of $\cup_{j\geq1} B_j$. Since $A_j\subseteq B_j$, the inequalities~\eqref{eq:epsOptimizers'} and \eqref{eq:coverControls'} yield that
  \begin{equation*}%
    F(s',y';\nu_j)\geq \varphi(s',y',n') - 3\eps\quad\mbox{for all }(s',y',n')\in A_j\cap \cD.
  \end{equation*}
  Similarly as above, we first fix $k\geq 1$, define the $\cF^t_\tau$-measurable sets
  \[
    \Gamma_j:=\big\{(\tau,X_{t,x}^\nu(\tau),M(\tau))\in A_j\big\}\quad\mbox{and}\quad
    \Gamma(k):=\bigcup_{1\leq j\leq k} \Gamma_j
  \]
  and set
  $
    \nu(k):=\big(\cdots\big(\big(\nu\otimes_{(\tau,\Gamma_1)} \nu_1\big)\otimes_{(\tau,\Gamma_2)} \nu_2\big)\cdots\otimes_{(\tau,\Gamma_k)}\nu_k\big).
  $
 To check that
  \[
    E\big[ f(X_{t,x}^{\nu(k)}(T)) \big|\cF_\tau\big] \geq F(\tau,X_{t,x}^\nu(\tau);\nu_j)\quad\mbox{on }\Gamma_j\quad\mbox{for } 1\leq j\leq k,
  \]
  we use that $\tau\leq s_j$ on $\Gamma_j$, so that we can apply~\eqref{eq:consistency} with the stopping time $\tilde{\tau}:=\tau\wedge s_j$ satisfying $\|\tilde{\tau}\|_{L^\infty}\leq s_j$ and thus $\nu_j\in\cU_{s_j}\subseteq \cU_{\|\tilde{\tau}\|_{L^\infty}}$; c.f.~\Bbiseqref{ass:B0'}. The rest of the proof is analogous to the above.
\end{proof}

\begin{remark}\label{rk:LindeloefAim}{\rm
  The assumption on $\sigma$-compactness in Theorem~\ref{th:DPP}(ii') was used only to ensure
  the Lindel\"of property of $\cD\cap\bD$ for the topology introduced in the proof. Therefore, any other assumption ensuring this will do as well.%
  }
\end{remark}

Let us record a slight generalization of Theorem~\ref{th:DPP}(ii),(ii') to the case of controls which are not necessarily admissible at the given point $(t,x,m)$. The intuition for this result is that the dynamic programming principle holds as before if we use such controls for a sufficiently short time (as formalized by condition~\eqref{eq:domainInvarianceRelax} below) and then switch to admissible ones. More precisely, the proof also exploits the relaxation which is anyway present in~\eqref{eq:DPPii}. We use the notation of Theorem~\ref{th:DPP}.

\begin{corollary}\label{co:DPPrelaxed}
  Let the assumptions of Theorem~\ref{th:DPP}(ii) hold true except that $\nu\in\cU_t$ and $M\in\cM_{t,m}$ are not necessarily in $\cU(t,x,m)$ and $\cM^+_{t,m,x}(\nu)$, respectively. In addition, assume that
  \begin{equation}\label{eq:domainInvarianceRelax}
    \big(\tau,X_{t,x}^\nu(\tau),M(\tau)\big)\in\bD\quad P\as
  \end{equation}
  Then the conclusion~\eqref{eq:DPPii} of Theorem~\ref{th:DPP}(ii) still holds true. Moreover, the same generalization holds true for Theorem~\ref{th:DPP}(ii').
\end{corollary}

\begin{proof}
  Let us inspect the proof of Theorem~\ref{th:DPP}(ii). Using directly~\eqref{eq:domainInvarianceRelax} rather than appealing to~\eqref{eq:domainInvariance}, the construction of the covering in Step~1 remains unchanged and the same is true for the concatenation in Step~2. In Step~3, we proceed as above up to and including~\eqref{eq:proofDPPMartOnGamma}.
  Note that~\eqref{eq:proofDPPuseAdmissibility} no longer holds as it used the assumption that $M\in\cM^+_{t,m,x}(\nu)$. However, in view of~\eqref{eq:proofDPPMartOnGamma} and $X_{t,x}^{\nu(k)}(T)=X_{t,x}^{\nu}(T)$ on $\Omega\setminus\Gamma(k)$, we still have
  \[
    E\big[g(X_{t,x}^{\nu(k)}(T))\big]\leq E\big[M^{(k)}(T)\1_{\Gamma(k)}\big] + E\big[g(X_{t,x}^{\nu}(T))\1_{\Omega\setminus\Gamma(k)}\big].
  \]
  Since $g(X_{t,x}^{\nu}(T))$ is integrable by~\eqref{eq:controlsIntegrable} and $\Gamma(k)\uparrow\Omega$, the latter expectation is bounded by $\delta$ for large $k$.
  Moreover, $\Gamma(k)\in\cF_\tau$, the martingale property of $M^{(k)}$, and the fact that $M^{(k)}=M+\delta$ on $[0,\tau]$ yield that
  \[
    E\big[M^{(k)}(T)\1_{\Gamma(k)}\big]=E\big[M^{(k)}(\tau)\1_{\Gamma(k)}\big] = E\big[(M(\tau)+\delta)\1_{\Gamma(k)}\big].
  \]
  Since $E[M(\tau)]=m$, the right hand side is dominated by $m+2\delta$ for large $k$.
  Together, we conclude that
  \[
    E[g(X_{t,x}^{\nu(k)}(T))]\leq m+3\delta;
  \]
  i.e., $\nu(k)\in\cU(t,x,m+3\delta)$ for all large $k$.
  Step~4 of the previous proof then applies as before (recall that $f(X_{t,x}^{\nu}(T))$ is integrable by~\eqref{eq:controlsIntegrable}), except that the changed admissibility of $\nu(k)$ now results in
  \[
    V(t,x,m+3\delta) \geq E[\varphi(\tau,X_{t,x}^\nu(\tau),M(\tau))].
  \]
  However, since $\delta>0$ was arbitrary, this is the same as~\eqref{eq:DPPii}. The argument to extend Theorem~\ref{th:DPP}(ii') is analogous.
\end{proof}

While we shall see that the relaxation $\delta>0$ in~\eqref{eq:DPPii} is harmless for the derivation of the Hamilton-Jacobi-Bellman equation, it is nevertheless interesting to know when the $\delta$ can be omitted; i.e., when $V$ is right continuous in $m$. The following sufficient condition will be used when we consider state constraints.

\begin{lemma}\label{lem: cond for right continuity in m}
  Let $(t,x,m)\in\bD$.
  For each $\delta>0$ there is $\nu^\delta\in\cU(t,x,m+\delta)$ such that
  \[
   F(t,x;\nu^\delta)\geq \delta^{-1}\wedge V(t,x,m+\delta)-\delta.
  \]
  Let $\delta_0>0$ and assume that for all $0<\delta\leq\delta_0$ there exists $\tilde{\nu}^\delta\in\cU(t,x,m)$ such that
  \[
    \lim_{\delta\downarrow0}P\big\{ X_{t,x}^{\nu^\delta}(T)\neq X_{t,x}^{\tilde{\nu}^\delta}(T) \big\}=0
  \]
  and such that the set
  \begin{equation}\label{eq: cond unif integrability f nu delta}
    \big\{\big[f(X_{t,x}^{\nu^\delta}(T))-f(X_{t,x}^{\tilde{\nu}^\delta}(T))\big]^+:\, 0<\delta\leq\delta_0\big\}\;\subseteq\; L^1(P)
  \end{equation}
  is uniformly integrable. Then $m'\mapsto V(t,x,m')$ is right continuous at $m$.
\end{lemma}

\begin{proof}
  Since $m'\mapsto\cU(t,x,m')$ is increasing, we have that $(t,x,m)\in\bD$ implies $\cU(t,x,m+\delta)\neq \emptyset$ for all $\delta \geq0$. Hence $\nu^\delta$ exists; of course, the truncation at $\delta^{-1}$ is necessary only if $V(t,x,m+\delta)=\infty$. Moreover, the monotonicity of $m'\mapsto V(t,x,m')$ implies that the right limit $V(t,x,m+)$ exists and that
  $V(t,x,m+)\geq V(t,x,m)$; it remains to prove the opposite inequality.
  Let $0<\delta\leq \delta_0$ and set $A^\delta:=\{ X_{t,x}^{\nu^\delta}(T)\neq X_{t,x}^{\tilde{\nu}^\delta}(T)\}$, then
  \begin{align*}
    V(t,x,m)
    & \geq F(t,x,\tilde{\nu}^\delta) \\
    & = F(t,x,\nu^\delta) - E\big[1_{A^\delta} \big(f(X_{t,x}^{\nu^\delta}(T))-f(X_{t,x}^{\tilde{\nu}^\delta}(T))\big) \big] \\
    & \geq \delta^{-1} \wedge V(t,x,m+\delta) - \delta - E\big[1_{A^\delta} \big(f(X_{t,x}^{\nu^\delta}(T))-f(X_{t,x}^{\tilde{\nu}^\delta}(T))\big)^+ \big].
  \end{align*}
  Letting $\delta\downarrow 0$, we deduce by~\eqref{eq: cond unif integrability f nu delta} that
  $V(t,x,m) \geq V(t,x,m+)$.
\end{proof}

\begin{remark}{\rm
  The integrability assumption~\eqref{eq: cond unif integrability f nu delta} is clearly satisfied if $f$ is bounded. In the general case, it may be useful to consider the value function for a truncated function $f$ in a first step.
  }
\end{remark}

\begin{remark}{\rm
  Our results can be generalized to a setting with multiple constraints. Given $N\in\N$ and $m\in\R^N$, let
  $$
    \cU(t,x,m):=\big\{\nu\in\cU_t:\, G^{i}(t,x;\nu)\leq m^{i}\mbox{ for } i=1,\dots, N \big\},
  $$
  where $G^{i}(t,x;\nu):=E[g^{i}(X_{t,x}^\nu(T))]$ for some measurable function $g^{i}$. In this case, $\cM_{t,0}$ is defined as the set  of c\`adl\`ag $N$-dimensional martingales   $M=\{M(s),\,s\in[t,T]\}$ with initial value $M(t)=0$, and
  \[
    \cM_{t,m}:=\{m+M:\, M\in\cM_{t,0}\},\quad m\in\R^{N}.
  \]
  This generalization, which has been considered in~\cite{BV12} within the framework of stochastic target problems with controlled loss,
  also allows to impose constraints at finitely many intermediate times $0\leq T_{1}\le T_{2}\le \cdots \le T$. Indeed, we can increase the dimension of the state process and add the components  $X_{t,x}^\nu(\cdot \wedge T_{j})$.
  }
\end{remark}

\section{Application to State Constraints}\label{sec: DPP state constraint}

We consider an open set $\cO\subseteq S:=\R^d$ and study the stochastic control problem under the constraint that the state process has to stay in $\cO$. Namely, we consider the value function
\begin{eqnarray}\label{eq: def V state constraint}
  \bar{V}(t,x):= \sup_{\nu\in\bar \cU(t,x)} F(t,x;\nu),\quad (t,x)\in\bS,
\end{eqnarray}
where
\[
  \bar{\cU}(t,x):=\big\{\nu\in\cU_t:\, X_{t,x}^\nu(s)\in \cO\mbox{ for all }s\in [t,T],\,P\as\big\}.
\]
In the following discussion we assume that, for all $(t,x)\in\bS$ and $\nu\in\cU_t$,
\begin{eqnarray}
  &\mbox{$X_{t,x}^\nu$ has continuous paths;}& \label{eq:contPaths}\\
  &\mbox{$(t,x)\mapsto X^\nu_{t,x}(r)$ is continuous in probability, uniformly in $r$;}& \label{eq:contInProbab}\\
  &\mbox{$\bar{\cU}(t,x)\neq\emptyset$ for $(t,x)\in[0,T]\times \cO$.}& \label{eq:fullDomain}
\end{eqnarray}
Explicitly, the condition~\eqref{eq:contInProbab} means that $(t_n,x_n)\to (t,x)$ implies
\[
  \sup_{r\in[0,T]} d\big(X^\nu_{t_n,x_n}(r),X^\nu_{t,x}(r)\big)\to0\quad\mbox{ in probability},
\]
where we set $X_{t,x}^\nu(r):=x$ for $r<t$,  $d(\cdot,\cdot)$ denotes the Euclidean metric, and it is implicitly assumed that $\nu\in\cU_{t_n}$ for all $n$.
We shall augment the state process so that the state constraint becomes a special case of an expectation constraint. To this end, we introduce the distance function
$d(x):=\inf \{d(x,x'):\,x'\in S\setminus\cO\}$ for $x\in S$
and the auxiliary process
\begin{equation}\label{eq:defY}
  Y^\nu_{t,x,y}(s):=y\wedge\inf_{r\in[t,s]} d(X_{t,x}^\nu(r)), \quad s\in[t,T],\quad y\in[0,\infty).
\end{equation}
By~\eqref{eq:contPaths}, each trajectory $\{X_{t,x}^\nu(r)(\omega),\,r\in[t,T]\}\subseteq S$ is compact; therefore, it has strictly positive distance to $S\setminus\cO$ whenever it is contained in $\cO$:
\[
  \{X_{t,x}^\nu(r)(\omega),\,r\in[t,T]\}\subseteq \cO \quad\mbox{if and only if}\quad Y^\nu_{t,x,1}(T)(\omega)>0.
\]
We consider the augmented state process
\[
  \bar{X}_{t,x,y}^\nu(\cdot):=\big(X_{t,x}^\nu(\cdot),Y_{t,x,y}^\nu(\cdot)\big)
\]
on the state space $S\times [0,\infty)$, then
\[
  E[g(\bar{X}_{t,x,y}^\nu(T))]=P\{Y^\nu_{t,x,y}(T)\leq 0\}\quad\mbox{by setting}\quad g(x,y):=\1_{(-\infty,0]}(y)
\]
for $(x,y)\in S\times [0,\infty)$. Now the state constraint may be expressed as $E[g(\bar{X}_{t,x,1}^\nu(T))]\leq0$
 and therefore
\[
  \bar \cU(t,x)=\cU(t,x,1,0) \quad\mbox{and}\quad \bar V(t,x)=V(t,x,1,0);
\]
of course, the value $1$ may be replaced by any number $y>0$. Here and in the sequel, we use the notation from the previous section applied to the controlled state process $\bar{X}$ on $S\times [0,\infty)$; i.e., we tacitly replace the variable $x$ by $(x,y)$ to define the set $\cU(t,x,y,m)$ of admissible controls and the associated value function $V(t,x,y,m)$.

One direction of the dynamic programming principle will be a consequence of the following counterpart of Assumption~\ref{ass:A}.

\begin{assumptionBar}\label{ass:Abar}
  For all $(t,x)\in\bS$, $\nu\in\bar \cU(t,x)$, $\tau\in\cT^t$ and $P$-a.e.\ $\omega\in\Omega$, there exists $\nu_\omega\in\bar \cU(\tau(\omega),X_{t,x}^\nu(\tau)(\omega))$ such that
  \[
    E\big[f(X_{t,x}^\nu(T))\big|\cF_\tau\big] (\omega) \leq F\big( \tau(\omega), X_{t,x}^\nu(\tau)(\omega);\nu_\omega \big).
  \]
\end{assumptionBar}

The more difficult direction of the dynamic programming principle will be inferred from Theorem~\ref{th:DPP} under a right-continuity condition; we shall exemplify in the subsequent section how to verify this condition.

\begin{theorem}\label{thm: DPP state constraint}
  Consider $(t,x)\in\bS$ and a family $\{\tau^\nu,\,\nu\in\bar \cU(t,x)\}\subseteq\cT^t$.
  \begin{enumerate}[topsep=3pt, partopsep=0pt, itemsep=1pt,parsep=2pt]
   \item Let Assumption~$\bar{A}$ hold true and let $\phi: \bS\to[-\infty,\infty]$ be a measurable function such that $\bar V\leq \phi$. Then $E \big[ \phi(\tau^\nu,X_{t,x}^\nu(\tau^\nu))^- \big]<\infty$ for all $\nu\in\bar \cU(t,x)$ and
      \begin{equation*}%
        \bar V(t,x) \leq \sup_{\nu\in\bar \cU(t,x)} E \big[\phi(\tau^\nu,X_{t,x}^\nu(\tau^\nu)) \big].
      \end{equation*}
  \item Let Assumption~\ref{ass:B'} hold true for the state process $\bar{X}$ on $S\times [0,\infty)$ and let
    \eqref{eq:contPaths}--\eqref{eq:fullDomain} hold true. Moreover, assume that
    \begin{equation}\label{eq:rightContCondition}
      V(t,x,1,0)=V(t,x,1,0+)
    \end{equation}
    and that $(s',x')\mapsto F(s',x';\nu)$ is l.s.c.\ on $[0,t_0]\times\cO$ for all $t_0\in[t,T]$ and $\nu\in\cU_{t_0}$.
    Then
    \begin{equation}\label{eq:DPPSCii}
      \bar V(t,x) \geq \sup_{\nu\in\bar \cU(t,x)} E \big[ \phi(\tau^\nu,X_{t,x}^\nu(\tau^\nu))\big]
    \end{equation}
    for any u.s.c.\ function $\phi: \bS\to[-\infty,\infty)$ such that $\bar V\geq \phi$.
  \end{enumerate}
\end{theorem}

\begin{proof}
  (i) We may assume that $\bar \cU(t,x)\neq\emptyset$ as otherwise $\bar{V}(t,x)=-\infty$. As in the proof of~\eqref{eq:DPPi}, we obtain that $F(t,x;\nu) \leq E[\varphi(\tau,X_{t,x}^\nu(\tau))]$ for all $\nu\in\bar \cU(t,x)$. The claim follows by taking supremum over $\nu$.

  (ii) Again, we may assume that $\bar \cU(t,x)\neq \emptyset$ as otherwise the right hand side of~\eqref{eq:DPPSCii} equals $-\infty$. We set
  \[
    \cD:=[t,T]\times \cO\times (0,\infty)\times\{0\}.
  \]
  If $\nu\in\bar \cU(t,x)$, then $g(\bar{X}_{t,x,1}^\nu(T))=0$ and hence the constant martingale $M:=0$ is contained in $\cM^+_{t,0,(x,1)}(\nu)$. Moreover, $Y_{t,x,1}^\nu(s)>0$ and hence $(s,\bar{X}_{t,x,1}^\nu(s),M(s))\in \cD$ for all $s\in [t,T]$, $P$-a.s.
  Furthermore, if we define
  \[
    \varphi(t',x',y',m') :=
    \begin{cases}
      \phi(t',x'), & (t',x',y',m') \in \cD \\
      -\infty, & \mbox{otherwise,}
    \end{cases}
  \]
  then $\varphi$ is u.s.c.\ on $\cD$. To see that the third semicontinuity condition in~\eqref{eq:semicontAss} is also satisfied, note that for any
  $(t',x',y'),(t'',x'',y'')\in [0,t_0]\times S\times[0,\infty)$ and $\nu\in\cU_{t_0}$,
  \begin{align*}
    |Y^\nu_{t',x',y'}(T)-&Y^\nu_{t'',x'',y''}(T)|\\
     &\leq |y'-y''|+ \Big|\inf_{r\in[t',T]} d(X_{t',x'}^\nu(r)) - \!\!\inf_{r\in[t'',T]} d(X_{t'',x''}^\nu(r))\Big|\\
     &\leq |y'-y''|+ \sup_{r\in[0,T]} d\big(X_{t',x'}^\nu(r),X_{t'',x''}^\nu(r)\big).
  \end{align*}
  Hence~\eqref{eq:contInProbab} implies that $(t',x',y')\mapsto Y^\nu_{t',x',y'}(T)$ is continuous in probability.
  As $(-\infty,0]\subset\R$ is closed, we conclude by the Portmanteau theorem that
  \[
    (t',x',y')\mapsto P\big\{Y^\nu_{t',x',y'}(T)\in (\infty,0]\big\}\equiv G(t',x',y';\nu)\quad\mbox{is u.s.c.}
  \]
  as required.
 Since any open subset of a Euclidean space is $\sigma$-compact and since~\eqref{eq:fullDomain} implies $\cD\cap\bD=\cD$, we can use \eqref{eq:rightContCondition} and Theorem~\ref{th:DPP}(ii') with $M=0$ to obtain that
  \begin{align*}
    \bar V(t,x)
     &=V(t,x,1,0)\\
     &=V(t,x,1,0+)\\
     &\geq E\big[\varphi\big(\tau^\nu,X_{t,x}^\nu(\tau^\nu),Y_{t,x,1}^\nu(\tau^\nu),0\big)\big]\\
     &=E [\phi(\tau^\nu,X_{t,x}^\nu(\tau^\nu)) ].
  \end{align*}
  As $\nu\in\bar \cU(t,x)$ was arbitrary, the result follows.
\end{proof}

\begin{remark}{\rm
  Similarly as in Theorem~\ref{th:DPP}(ii), there is also a version of Theorem~\ref{thm: DPP state constraint}(ii) for stopping times taking countably many values. In this case, Assumption~\ref{ass:B} replaces Assumption~\ref{ass:B'}, all conditions on the time variable are superfluous, and one can consider a general separable metric space $S$.
  }
\end{remark}

\section{Application to Controlled Diffusions}\label{se:ApplControlledDiff}

In this section, we show how the weak formulation of the dynamic programming principle applies in the context of controlled Brownian stochastic differential equations and how it allows to derive the Hamilton-Jacobi-Bellman PDEs for the value functions associated to optimal control problems with expectation or state constraints.
As the main purpose of this section is to illustrate the use of Theorems~\ref{th:DPP} and~\ref{thm: DPP state constraint}, we shall choose a fairly simple setup allowing to explain the main points without too many distractions.
Given the generality of those theorems, extensions such as singular control, mixed control/stopping problems, etc.\ do not present any particular difficulty.

\subsection{Setup for Controlled Diffusions}\label{sec: dynamics}

From now on, we take $S=\R^d$ and let $\Omega=C([0,T];\R^d)$ be the space of continuous paths, $P$ the Wiener measure on $\Omega$, and $W$ the canonical process $W_{t}(\omega)=\omega_{t}$. Let $\F=(\cF_t)_{t\in[0,T]}$ be the augmentation of the filtration generated by $W$; without loss of generality, $\cF=\cF_T$. For $t\in[0,T]$, the auxiliary filtration $\F^t=(\cF^t_s)_{s\in[0,T]}$ is chosen to be the augmentation of $\sigma(W_r-W_t,\,t\leq r\leq s)$; in particular, $\F^t$ is independent of $\cF_t$.

Consider a closed set $U\subseteq \R^{d}$ and let $\cU$ be the set of all $U$-valued predictable processes $\nu$ satisfying
$E[\int_0^T|\nu_t|^2\,dt]<\infty$. Then we set
\[
  \cU_t=\big\{\nu\in\cU:\, \nu\mbox{ is $\F^t$-predictable}\big\}.
\]
This choice will be convenient to verify Assumption~B'. We remark that the restriction to $\F^t$-predictable controls entails no loss of generality, in the sense that the alternative choice $\cU_t=\cU$ would result in the same value function. Indeed, this follows from a well known randomization argument (see, e.g., \cite[Remark~5.2]{BT10}).

Let $\M^{d}$ denote the set of $d\x d$ matrices. Given two Lipschitz continuous functions
\[
  \mu:\R^{d}\x U\to \R^{d}, \quad\sigma:\R^{d}\x U\to \M^{d}
\]
and $(t,x,\nu)\in [0,T]\x\R^{d}\x\cU$, we define $X_{t,x}^{\nu}(\cdot)$ as the unique strong solution of the stochastic differential equation (SDE)
\begin{equation}\label{eq: def X eds}
X(s)=x+\int_{t}^{s} \mu(X(r),\nu_{r})\,dr + \int_{t}^{s} \sigma(X(r),\nu_{r})\,dW_{r},\quad t\le s \le T,
\end{equation}
where we set $X_{t,x}^{\nu}(r)=x$ for $r\leq t$.
As $X^{\nu}_{t,x}(T)$ is square integrable for any $\nu\in \cU$, \eqref{eq:controlsIntegrable} is satisfied whenever
\begin{equation}\label{eq: cond growth f and g}
\mbox{$f$ and $g$ have quadratic growth,}
\end{equation}
which we assume from now on. In addition, we also impose that
\begin{equation}\label{eq: semi cont f and g}
  \begin{cases}
    & \mbox{$f$ is l.s.c.\ and $f^-$ has subquadratic growth}, \\
    & \mbox{$g$ is u.s.c.\ and $g^+$ has subquadratic growth},
  \end{cases}
\end{equation}
where $h: \R^d\to\R$ is said to have subquadratic growth if $h(x)/|x|^2\to 0$ as $|x|\to\infty$.
This will be used to obtain the semicontinuity properties~\eqref{eq:semicontAss}.

Furthermore, we take  $\cM_{t,0}$ to be the family of all c\`adl\`ag martingales which start at $0$ and are adapted to $\F^t$. By the independence of the increments of $W$, we see that $M\in \cM_{t,0}$ is then also a martingale in the filtration $\F^t$ and that $M_r=0$ for $r\leq t$. For $\nu\in\cU_t$, we have $X^\nu_{t,x}(T)\in L^2(\cF^t_T,P)$ and hence~\eqref{eq:martRichness} is satisfied with $M^\nu_t[x](\cdot)=E[X^\nu_{t,x}(T)|\cF^t_\cdot]$.
It will be useful to express the martingales as stochastic integrals. Let
$\cA_{t}$ denote the set of $\R^{d}$-valued $\F^t$-predictable processes $\alpha$ such that $\int_0^T|\alpha_t|^2\,dt<\infty$ $P$-a.s.\ and such that
\[
  M^{\alpha}_{t,0}(\cdot):=\int_{t}^{\cdot }\alpha_{s}^{\top}\,dW_{s}
\]
is a martingale (${}^{\top}$ denotes transposition). Then the Brownian representation theorem yields that
\[
  \cM_{t,0}=\left\{  M^{\alpha}_{t,0}: \;\alpha \in \cA_{t}\right\}.
\]
In the following, we also use the notation $M_{t,m}^{\alpha}:=m+ M_{t,0}^{\alpha}$.

\begin{lemma}\label{lem: verification assumption section dynamics}
  In the above setup, Assumptions~A, $\bar{A}$, B, B' are satisfied and $F$ and $G$ satisfy~\eqref{eq:semicontAss}.
\end{lemma}

\begin{proof}
  Assumption~(B0') is immediate from the definition of $\cU_t$. We define the concatenation of controls by~\eqref{eq:concatenationExample}.

  The validity of Assumptions~A, ${\rm \bar{A}}$ and (B1') follows from the uniqueness and the flow property of~\eqref{eq: def X eds}; in particular, the control $\nu_\omega$ in Assumption~A can be defined by  $\nu_\omega(\omega'):=\nu(\omega\otimes_{\tau}\omega')$, $\omega'\in\Omega$, where the concatenated path $\omega\otimes_{\tau}\omega'$ is given by
  \[
    (\omega\otimes_{\tau}\omega')_r:=\omega_r\1_{[0,\tau(\omega)]}(r) + \big(\omega'_r-\omega'_\tau(\omega)+\omega_\tau(\omega)\big)\1_{(\tau(\omega),T]}(r).
  \]
  While we refer to \cite[Proposition 5.4]{BT10} for a detailed proof, we emphasize that the choice of $\cU_{s}$ is crucial for the validity of~\eqref{eq:consistency}: in the notation of Assumption~B',~\eqref{eq:consistency} essentially requires that $\bar\nu$ be independent of $\cF_\tau$.

  Let $t,\tau,\nu,\bar{\nu}$ be as in Assumption~B', we show that (B2') holds.
  Let $\bar{M}$ be a c\`adl\`ag version of
  \[
   \bar M(r):=E[g(X_{t,x}^{\hat \nu}(T))|\cF_r],\; r\in[0,T],\quad\mbox{where }\hat \nu:=\nu\1_{[0,\tau]}+\1_{(\tau,T]} \bar \nu.
  \]
  By the same argument as in~\cite[Proposition 5.4]{BT10}, we deduce from the uniqueness and the flow
  property of~\eqref{eq: def X eds} and the fact that $\bar \nu$ is independent of $\cF_\tau$ that
  \[
    E[g(X_{t,x}^{\hat  \nu}(T))|\cF_{r}]=E[g(X_{\tau,X_{t,x}^{\nu}(\tau)}^{\bar\nu}(T))|\cF_{r}]=M^{\bar \nu}_{\tau}[X_{t,x}^{\nu}(\tau)](r)\quad \mbox{on }[\tau,T].
  \]
  Hence $\bar M=M^{\bar \nu}_{\tau}[X_{t,x}^{\nu}(\tau)]$ on $[\tau,T]$. The last assertion of (B2') is clear by the definition of $\cM_{t,0}$. As already mentioned in Remark~\ref{rem: ass A implies B3}, Assumption (B3') follows from Assumption A and Assumption~B follows from Assumption~B'.

  Next, we check that $F$ satisfies~\eqref{eq:semicontAss}; i.e., that $F$ is l.s.c. For fixed $\nu\in\cU$,
  $(t,x)\mapsto X^\nu_{t,x}(T)$ is $L^2$-continuous.
  Hence the semicontinuity from~\eqref{eq: semi cont f and g} and Fatou's lemma yield that $(t,x)\mapsto E[f(X^\nu_{t,x}(T))^+]$ is l.s.c. By the subquadratic growth from~\eqref{eq: semi cont f and g}, we have that
  $\{f(X^\nu_{t,x}(T))^-:\, (t,x)\in B\}$ is uniformly integrable whenever $B\subset\bS$ is bounded, hence the semicontinuity of $f$ also yields that $(t,x)\mapsto E[f(X^\nu_{t,x}(T))^-]$ is u.s.c. As a result, $F$ is l.s.c. The same arguments show that $G$ also satisfies~\eqref{eq:semicontAss}.
\end{proof}

\subsection{PDE for Expectation Constraints}

In this section, we show how to deduce the Hamilton-Jacobi-Bellman PDE for the optimal control problem~\eqref{eq:valueFctDef} with expectation constraint from the weak dynamic programming principle stated in Theorem~\ref{th:DPP}. Given a suitably differentiable function $\vp(t,x)$ on $[0,T]\x \R^{d}$, we shall denote by $\partial_{t} \vp$ its derivative with respect to $t$ and by $D\vp$ and $D^{2}\vp$ the Jacobian and the Hessian matrix with respect to $x$, respectively.

In the context of the setup introduced in the preceding Section~\ref{sec: dynamics}, the Hamilton-Jacobi-Bellman operator is given by
\[
  H(x,p,Q):=\inf_{(u,a)\in U \x \R^{d}} \big(-L^{u,a}(x,p,Q)\big),\quad (x,p,Q)\in \R^{d}\x \R^{d+1}\x \M^{d+1},
\]
where
\[
  L^{u,a}(x,p,Q):= \mu_{X,M}(x,u)^{\top} p +\frac12 \tr[\sigma_{X,M}\sigma_{X,M}^{\top}(x,u,a)Q],\quad (u,a)\in U\x \R^{d}
\]
is the Dynkin operator with coefficients
\[
  \mu_{X,M}(x,u):=\left(\begin{array}{c}\mu(x,u)\\ 0\end{array}\right)
  \quad\mbox{and}\quad
  \sigma_{X,M}(x,u,a):=\left(\begin{array}{c}\sigma(x,u)\\ a^{\top}\end{array}\right).
\]
Since the set $U\x \R^{d}$ is unbounded, $H$ may be discontinuous and
viscosity solution properties need to be stated in terms of the upper and lower semicontinuous envelopes of $H$,
\begin{align*}
  H^{*}(x,p,Q)&:=\limsup\limits_{(x',p',Q')\to (x,p,Q)} H(x',p',Q'),\\
  H_{*}(x,p,Q)&:=\liminf \limits_{(x',p',Q')\to (x,p,Q)} H(x',p',Q').
\end{align*}
The value function $V$ defined in \eqref{eq:valueFctDef} may also be discontinuous and so we introduce
\begin{align*}
  V^{*}(t,x,m):=\limsup \limits_{\tiny \begin{array}{c}(t',x',m')\to (t,x,m)\\ (t',x',m')\in \Int\bD\end{array}} V(t',x',m'),\\
  V_{*}(t,x,m):=\liminf \limits_{\tiny \begin{array}{c}(t',x',m')\to (t,x,m)\\ (t',x',m')\in \Int\bD\end{array}} V(t',x',m').
\end{align*}
Here $\Int\bD$ denotes the parabolic interior; i.e, the interior of $\bD\setminus\{t=T\}$ in $\hbS$, where
$\{t=T\}:=\{(t,x,m)\in\hbS:\,t=T\}$. Moreover, we shall denote by $\overline{\bD}$ the closure of $\bD$.
The main result of this subsection is the following PDE.

\begin{theorem}\label{thm: PDE characterization V}
  Assume that $V$ is locally bounded on $\Int\bD$.
  \begin{enumerate}[topsep=3pt, partopsep=0pt, itemsep=1pt,parsep=2pt]
   \item The function $V^{*}$ is a viscosity subsolution on $\overline{\bD}\setminus\{t=T\}$ of
   \[
    -\partial_{t} \vp+H_{*}(\cdot,D\vp,D^{2}\vp)\le 0.
   \]
   \item The function $V_{*}$ is a viscosity supersolution on $\Int\bD$ of
   \[
    -\partial_{t} \vp+H^{*}(\cdot,D\vp,D^{2}\vp)\ge 0.
   \]
  \end{enumerate}
\end{theorem}

We refer to~\cite{CrIsLi92} for the various equivalent definitions of viscosity super- and subsolutions. We merely mention that ``subsolution on $A$'' means that the subsolution property is satisfied at points of $A$ which are local maxima of $V^{*}-\vp$ on $A$, where $\vp$ is a test function, and analogously for the supersolution.

We shall not discuss in this generality the boundary condition and the validity of a comparison principle. In the subsequent section, these will be studied in some detail for the case of state constraints. We also refer to \cite{BEI10} for the study of the boundary conditions in a similar framework.

\begin{remark}\label{rem: function v}{\rm
  We observe that the domain of the PDE in Theorem~\ref{thm: PDE characterization V} is not given a priori;
  it is itself characterized by a control problem: if we define
  \begin{equation}\label{eq:defv}
   v(t,x):=\inf_{{\nu \in \cU_{t}}} E[g(X^{\nu}_{t,x}(T))],\quad (t,x)\in\bS,
  \end{equation}
  then
  \[%
   \Int\bD=\big\{(t,x,m)\in \hbS:\,m>v^{*}(t,x),\, t<T\big\},
  \]
  where $v^{*}$ is the upper semicontinuous envelope of $v$ on $[0,T)\x \R^{d}$. In particular, $\Int\bD\neq\emptyset$ since $v$ is locally bounded from above.
  In fact, in the present setup, we also have
  \begin{equation}\label{eq:ProofDomainInClosure}
   \Int\bD=\big\{(t,x,m)\in \hbS:\,m>v(t,x),\, t<T\big\}.
  \end{equation}
  Indeed, a well known randomization argument (e.g., \cite[Remark~5.2]{BT10}) yields that
  $v(t,x)=\inf_{{\nu \in \cU}} E[g(X^{\nu}_{t,x}(T))]$ for all $(t,x)\in\bS$; i.e., the set $\cU_t$ in~\eqref{eq:defv} can be replaced by $\cU$. Therefore, $v$ inherits the upper semicontinuity of $G$
  (c.f.\ Lemma~\ref{lem: verification assumption section dynamics}) and we have $v=v^*$.
  Using~\eqref{eq:ProofDomainInClosure}, we obtain that
  \[
   \big\{(t,x,m)\in \hbS:\,m\ge v(t,x)\big\}
   \subseteq \big\{(t,x,m)\in \hbS:\,m\ge v_*(t,x)\big\} = \overline{\Int\bD},
  \]
  where $v_{*}$ is the lower semicontinuous envelope of $v$ on $[0,T]\x \R^{d}$, and hence
  \begin{equation}\label{eq:DomainInClosure}
    \bD\subseteq \overline{\Int\bD}.
  \end{equation}
  If furthermore $v$ is continuous and the infimum in~\eqref{eq:defv} is attained for all $(t,x)$, then
  the converse inclusion is also satisfied. In applications, it may be desirable to have continuity of $v$ so that
  $\Int \bD$ and $\overline{\bD}$ are described directly by $v$ (rather than a semicontinuous envelope). To analyze the continuity of $v$, one can study the comparison principle for the Hamilton-Jacobi-Bellman equation associated with the control problem~\eqref{eq:defv}. We refer to \cite[Sections~5~and~6]{follmer2000efficient} for the explicit computation of $v$  in an example from Mathematical Finance.
  }
\end{remark}

The rest of this subsection is devoted to the proof of Theorem~\ref{thm: PDE characterization V}. We first state a version of Theorem~\ref{th:DPP} which is suitable to derive the PDE.

\begin{lemma}\label{lem: DPP constraint in expectation}
  (i) Let $B$ be an open neighborhood of a point $(t,x,m)\in\bD$ such that $V(t,x,m)<\infty$ and let $\vp: \overline{B} \to \R$ be a continuous function such that $V\le \vp$ on $\overline{B}$. For all $\eps>0$ there exist $(\nu,\alpha)\in \cU_{t}\x \cA_{t}$ such that
  \[
    V(t,x,m)\le E \big[ \varphi(\tau,X_{t,x}^\nu(\tau),M_{t,m}^{\alpha}(\tau)) \big]+\eps\quad\mbox{and}\quad M_{t,m}^{\alpha}(T)\ge g(X_{t,x}^\nu(T)),
  \]
  where $\tau$ is the first exit time of  $(s,X_{t,x}^{\nu}(s),M_{t,m}^{\alpha}(s))_{s\ge t}$ from $B$.

  (ii) Let $B$ be an open neighborhood of a point $(t,x,m)\in\Int\bD$ such that $\overline{B}\subseteq \bD$ and let
  $(\nu,\alpha)\in\cU_t\times\cA_t$. For any continuous function $\varphi: \overline{B} \to \R$ satisfying $V\ge  \varphi$ on $\overline{B}$ and for any $\eps>0$,
  \begin{equation}\label{eq: DPP constraint in expectation ii}
    V(t,x,m+\eps) \geq E \big[ \varphi(\tau,X_{t,x}^\nu(\tau),M_{t,m}^{\alpha}(\tau)) \big],
  \end{equation}
  where $\tau$ is the first exit time of $(s,X_{t,x}^{\nu}(s),M_{t,m}^{\alpha}(s))_{s\ge t}$ from $B$.
\end{lemma}

\begin{proof}
  In view of Lemmata~\ref{le:reformulationOfConstraint} and~\ref{lem: verification assumption section dynamics}, part~(i) is immediate from Theorem~\ref{th:DPP}(i).

  For part~(ii) we use the extension of Theorem~\ref{th:DPP}(ii') as stated in Corollary~\ref{co:DPPrelaxed} with $\cD:=\overline B$. Note that $(\tau,X_{t,x}^\nu(\tau),M_{t,x}^\alpha(\tau))\in\overline B\subseteq \bD$; in particular,
  $\cD\cap\bD=\cD$ is closed and hence $\sigma$-compact.
\end{proof}

We can now deduce the PDE for $V$ from the dynamic programming principle in the form of Lemma~\ref{lem: DPP constraint in expectation}. Although the arguments are the usual ones, we shall indicate the proof, in particular to show that the relaxation ``$m+\eps$'' in~\eqref{eq: DPP constraint in expectation ii} does not affect the PDE.

\begin{proof}[Proof of Theorem~\ref{thm: PDE characterization V}]
  (i)~We first prove the subsolution property. Let $\vp$ be a $C^{1,2}$-function and let $(t_{0},x_{0},m_{0})\in \overline \bD$ be such that $t_0\in(0,T)$ and $(t_{0},x_{0},m_{0})$ is a  maximum point of $V^{*}-\vp$ satisfying
  \begin{equation}\label{eq : max V-vp = 0}
    (V^{*}-\vp)(t_{0},x_{0},m_{0})=0.
  \end{equation}
  Assume for contradiction that
  \[
    \big(-\partial_{t} \vp+H_{*}(\cdot,D\vp,D^{2}\vp)\big)(t_{0},x_{0},m_{0})>0.
  \]
   Since  $ \overline \bD=\overline{\Int\bD}$ by~\eqref{eq:DomainInClosure}, there exists a bounded open neighborhood $B\subset \hbS$ of $(t_{0},x_{0},m_{0})$  such that
  \begin{equation}\label{eq: contra sub sol pde}
   -\partial_{t} \bar \vp-L^{  u,  a}(\cdot,D\bar \vp,D^{2}\bar \vp) >0 \quad\mbox{on } \overline{B}\cap \overline{\Int\bD},  \quad\mbox{for all } (u,a)\in U\x \R^{d},
  \end{equation}
  where
  \[
    \bar \vp(s,y,n):=\vp(t_{0},x_{0},m_{0})+\left(|s-t_{0}|^{2}+|y-x_{0}|^{4}+|n-m_{0}|^{4}\right).
  \]
  Moreover, we have
  \begin{equation}\label{eq: max strict for sub sol}
    \eta:=\min_{\partial B} (\bar\vp-\vp)>0.
  \end{equation}
  Given $\eps>0$, let $(t_{\eps},x_{\eps},m_{\eps})\in  B\cap \Int\bD$ be such that
  \begin{equation}\label{eq : def point eps for sub sol}
  V(t_{\eps},x_{\eps},m_{\eps})\ge V^{*}(t_{0},x_{0},m_{0})-\eps.
  \end{equation}
  Consider arbitrary $(\nu,\alpha)\in \cU_{t}\x \cA_{t}$ such that $M_{t_{\eps},m_{\eps}}^{\alpha}(T)\ge g(X_{t_{\eps},x_{\eps}}^{\nu}(T))$ and let $\tau$ be the first exit time of $(s,X_{t_{\eps},x_{\eps}}^{\nu}(s),M_{t_{\eps},m_{\eps}}^{\alpha}(s))_{s\ge t}$ from $B$.
  We recall from Remark~\ref{rem: ass A implies B3}(ii) that $(s,X_{t_{\eps},x_{\eps}}^{\nu}(s),M_{t_{\eps},m_{\eps}}^{\alpha}(s))_{s\ge t}$ remains in $\bD$ on $[t,T]$, and hence also in $\overline{\Int \bD}$ by~\eqref{eq:DomainInClosure}. Now, it follows from It\^{o}'s formula and~\eqref{eq: contra sub sol pde} that
  \begin{equation*}%
    \bar \vp(t_{\eps},x_{\eps},m_{\eps})\geq E\big[\bar \vp(\tau,X_{t_{\eps},x_{\eps}}^{\nu}(\tau),M_{t_{\eps},m_{\eps}}^{\alpha}(\tau))\big].
  \end{equation*}
  For $(t_{\eps},x_{\eps},m_{\eps})$ close enough to $(t_{0},x_{0},m_{0})$, this implies that
  \[
    \bar \vp(t_{0},x_{0},m_{0})\ge E\big[\bar \vp(\tau,X_{t_{\eps},x_{\eps}}^{\nu}(\tau),M_{t_{\eps},m_{\eps}}^{\alpha}(\tau))\big]-o(1) \quad\mbox{as }\eps\to0,
  \]
  which, by \eqref{eq : max V-vp = 0}, \eqref{eq: max strict for sub sol} and  \eqref{eq : def point eps for sub sol},    leads to
  \[
    V(t_{\eps},x_{\eps},m_{\eps})  \ge E\big[\vp(\tau,X_{t_{\eps},x_{\eps}}^{\nu}(\tau),M_{t_{\eps},m_{\eps}}^{\alpha}(\tau))\big]+\eta-o(1).
  \]
  This contradicts Lemma~\ref{lem: DPP constraint in expectation}(i) for $\eps>0$ small enough.

  (ii)~We now prove the supersolution property.  Let $\vp$ be a $C^{1,2}$-function and let $(t_{0},x_{0},m_{0})\in \Int\bD$ be such that $(t_{0},x_{0},m_{0})$ is a  minimum point of $V_{*}-\vp$ satisfying
  \begin{equation}\label{eq : min V-vp = 0}
    (V_{*}-\vp)(t_{0},x_{0},m_{0})=0.
  \end{equation}
  Assume for contradiction that
  \[
    \big(-\partial_{t} \vp+H^{*}(\cdot,D\vp,D^{2}\vp)\big)(t_{0},x_{0},m_{0})<0.
  \]
  Then there exist $(\hat u,\hat a)\in U\x \R^{d}$ and a bounded open neighborhood $B$ of $(t_{0},x_{0},m_{0})$ such that $\overline{B}\subseteq \Int\bD$ and
  \begin{equation}\label{eq: contra sur sol pde}
   -\partial_{t} \bar \vp-L^{\hat u,\hat a}(\cdot,D\bar \vp,D^{2}\bar \vp) <0 \quad\mbox{on } B,
  \end{equation}
  where
  \[
    \bar \vp(s,y,n):=\vp(t_{0},x_{0},m_{0})-\big(|s-t_{0}|^{2}+|y-x_{0}|^{4}+|n-m_{0}|^{4}\big).
  \]
  Note that
  \begin{equation}\label{eq: min strict for sur sol}
    \eta:=\min_{\partial B} (\vp-\bar \vp)>0.
  \end{equation}
  Given $\eps>0$, let $(t_{\eps},x_{\eps},m_{\eps})\in  B$ be such that
  \begin{equation}\label{eq : def point eps for super sol}
    V(t_{\eps},x_{\eps},m_{\eps}+\eps)\le V_{*}(t_{0},x_{0},m_{0})+\eps.
  \end{equation}
  Viewing $(\hat u,\hat a)$ as a constant control, it follows from It\^{o}'s formula and \eqref{eq: contra sur sol pde} that
  \[
    \bar \vp(t_{\eps},x_{\eps},m_{\eps})\le E\big[\bar \vp(\tau,X_{t_{\eps},x_{\eps}}^{\hat u}(\tau),M_{t_{\eps},m_{\eps}}^{\hat a}(\tau))\big],
  \]
  where  $\tau$ is the first exit time of $(s,X_{t_{\eps},x_{\eps}}^{\hat u}(s),M_{t_{\eps},m_{\eps}}^{\hat a}(s))_{s\geq t}$ from $B$. For $(t_{\eps},x_{\eps},m_{\eps})$ close enough to $(t_{0},x_{0},m_{0})$, this implies that
  \[
    \bar \vp(t_{0},x_{0},m_{0})\le E\big[\bar \vp(\tau,X_{t_{\eps},x_{\eps}}^{\hat u}(\tau),M_{t_{\eps},m_{\eps}}^{\hat a}(\tau))\big]+o(1)\quad\mbox{as }\eps\to0,
  \]
  which, by \eqref{eq : min V-vp = 0}, \eqref{eq: min strict for sur sol} and  \eqref{eq : def point eps for super sol},    leads to
  \[
    V(t_{\eps},x_{\eps},m_{\eps}+\eps)  \le E\big[\vp(\tau,X_{t_{\eps},x_{\eps}}^{\hat u}(\tau),M_{t_{\eps},m_{\eps}}^{\hat a}(\tau))\big]-\eta+o(1).
  \]
  For $\eps>0$ small enough, this yields a contradiction to Lemma~\ref{lem: DPP constraint in expectation}(ii).
\end{proof}

\subsection{PDE for State Constraints}\label{subsec: pde for state constraint}

In this section, we discuss the Hamilton-Jacobi-Bellman PDE for the state constraint problem (cf.\ Section~\ref{sec: DPP state constraint}) in the case where the state process is given by a controlled SDE as introduced in Section \ref{sec: dynamics} and required to stay in an open set $\cO\subseteq\R^d$. Note that in this setup, the continuity conditions~\eqref{eq:contPaths} and~\eqref{eq:contInProbab} are satisfied.

We shall derive the PDE via Theorem~\ref{thm: DPP state constraint}. The basic idea to guarantee its condition~\eqref{eq:rightContCondition} about right continuity in the constraint level runs as follows. Consider a control $\nu\in\cU_t$ such that $X^\nu_{t,x}$ leaves $\cO$ with at most small probability $\delta\geq0$. Then we shall construct a control $\nu^\delta$, satisfying the state constraint, by switching to some admissible control $\hat{\nu}$ shortly before $X^\nu_{t,x}$ exits $\cO$. As a result, $\nu^\delta$ coincides with $\nu$ on a set of large probability and therefore the reward is similar. Along the lines of Lemma~\ref{lem: cond for right continuity in m} we shall then obtain the desired right continuity (cf.\ Lemma~\ref{lem: V is right continuous in m in the SDE model} below).

To make this work,
we clearly need to have $\bar{\cU}(t,x)\neq\emptyset$ for all $(t,x)$ in $[0,T]\times \cO$, which is anyway necessary for the value function $\bar V$ from~\eqref{eq: def V state constraint} to be finite. However, we need a slightly stronger condition; namely, that we can switch to an admissible control in a measurable way. A particularly simple condition ensuring this, is the existence of an admissible feedback control:

\begin{assumption}\label{ass:C}
 There exists a Lipschitz continuous mapping $\hat u: \cO \to U$ such that, for all $(t,x)\in [0,T]\x \cO$, the solution $\hat X_{t,x}$ of
 \begin{equation}\label{eq: SDE hat X}
 \hat X(s)=x+\int_{t}^{s} \mu\big(\hat X(r),\hat u(\hat X(r))\big)\,dr +\int_{t}^{s} \sigma\big(\hat X(r),\hat u(\hat X(r))\big)\,dW_{r},\quad s\in [t,T]
 \end{equation}
 satisfies $\hat X_{t,x}(s)\in \cO$ for all $s\in [t,T]$, $P$-a.s.
\end{assumption}

If, e.g., $\mu(\cdot,u_0)=0$ and $\sigma(\cdot,u_0)=0$ for some $u_0\in U$, then Assumption~C is clearly satisfied for $\hat{u}\equiv u_0$. Or, under an additional smoothness condition, $\hat X_{t,x}$ will stay in $\cO$ if the Lipschitz function $\hat{u}$ satisfies
\begin{eqnarray*}
| n \sigma|(\cdot,\hat u)=0 \quad \mbox{and}\quad  \left(n^{\top}\mu +\frac12 {\rm Trace}[D n\;\sigma\sigma^{\top}] \right)(\cdot,\hat u)>0 \quad \mbox{on}\quad \partial \cO,
\end{eqnarray*}
where $n$ denotes the inner normal to $\partial \cO$; see also \cite[Proposition~3.1]{Ka94} and \cite[Lemma~III.4.3]{freidlin1985functional}.

The following is a simple condition guaranteeing the uniform integrability required in~\eqref{eq: cond unif integrability f nu delta}.

\begin{assumption}\label{ass:D}
  Either $f$ is bounded or the coefficients $\mu(x,u)$ and $\sigma(x,u)$ in the SDE~\eqref{eq: def X eds} have linear growth in $x$, uniformly in $u$.
\end{assumption}

This assumption holds in particular if the control domain $U$ is bounded.

\begin{remark}{\rm
  Assumption~\ref{ass:C} implies that $\bar V$ is locally bounded from below and Assumption~\ref{ass:D} implies that $\bar V$ is locally bounded from above. %
  }
\end{remark}

Next, we introduce the notation for the PDE related to the value function $\bar V$ from \eqref{eq: def V state constraint}.
The associated Hamilton-Jacobi-Bellman operator is given by
\begin{equation}\label{eq: def bar H example}
  \bar H(x,p,Q):=\inf_{u \in U} \big(-\bar L^{u}(x,p,Q)\big),\quad (x,p,Q)\in \R^{d}\x \R^{d}\x \M^{d},
\end{equation}
where the Dynkin operator is defined by
\[
  \bar{L}^{u}(x,p,Q):= \mu(x,u)^{\top} p +\frac12 \tr[\sigma\sigma^{\top}(x,u)Q], \quad u\in U.
\]
Similarly as above, we introduce the semicontinuous envelopes
\begin{align*}
  \bar H^{*}(x,p,Q)&:=\limsup \limits_{\tiny \begin{array}{c}(x',p',Q')\to (x,p,Q)\\ (x',p',Q')\in \cO\x\R^d\x\M^d\end{array}} \bar H(x',p',Q'),\\
  \bar H_{*}(x,p,Q)&:=\liminf \limits_{\tiny \begin{array}{c}(x',p',Q')\to (x,p,Q)\\ (x',p',Q')\in \cO\x\R^d\x\M^d\end{array}} \bar H(x',p',Q')
\end{align*}
and
\begin{align*}
  \bar V^{*}(t,x):=\limsup \limits_{\tiny \begin{array}{c}(t',x')\to (t,x)\\ (t',x')\in [0,T)\x\cO\end{array}} \bar V(t',x'),\\
  \bar V_{*}(t,x):=\liminf \limits_{\tiny \begin{array}{c}(t',x')\to (t,x)\\ (t',x')\in [0,T)\x\cO\end{array}} \bar V(t',x').
\end{align*}

We can now state the Hamilton-Jacobi-Bellman PDE; the boundary condition is discussed in
Proposition~\ref{prop: terminal cond and uniqueness state constraint} below.

\begin{theorem}\label{thm: PDE state constraint}
  Assume that $\bar V$ is locally bounded on $[0,T)\times\cO$.
  \begin{enumerate}[topsep=3pt, partopsep=0pt, itemsep=1pt,parsep=2pt]
   \item The function $\bar V^{*}$ is a viscosity subsolution on $[0,T)\times\overline{\cO}$ of
   \begin{equation}\label{eq: PDE subsol thm state constraint}
    -\partial_{t} \vp+\bar H_{*}(\cdot,D\vp,D^{2}\vp)\le 0.
   \end{equation}
   \item Under Assumptions~\ref{ass:C} and~\ref{ass:D}, the function $\bar V_{*}$ is a viscosity supersolution on $[0,T)\times\cO$ of
   \begin{equation*}%
    -\partial_{t} \vp+\bar H^{*}(\cdot,D\vp,D^{2}\vp)\ge 0.
   \end{equation*}
  \end{enumerate}
\end{theorem}

The proof is given below, after some auxiliary results. We first verify the right-continuity condition~\eqref{eq:rightContCondition} for the value function $V(t,x,y,m)$ introduced in Section~\ref{sec: DPP state constraint}.

\begin{lemma}\label{lem: V is right continuous in m in the SDE model}
  Let Assumptions \ref{ass:C} and \ref{ass:D} hold true. Then $V(t,x,1,0+)=V(t,x,1,0)$ for all $(t,x)\in [0,T]\x \cO$.
\end{lemma}

\begin{proof}
  For $\delta>0$, let $\nu^{\delta}\in \cU(t,x,\delta)$ be as in Lemma~\ref{lem: cond for right continuity in m}. Then the process $Y_{t,x,y}^{\nu^{\delta}}$ defined in~\eqref{eq:defY} satisfies  $Y_{t,x,1}^{\nu^{\delta}}(T)>0$ outside of a set of measure at most $\delta$. It follows that we can find $\eps_{\delta}\in (0,1\wedge d(x))$ such that
  the set $A^{\delta}:=\{Y_{t,x,1}^{\nu}(T)\le \eps_{\delta}\}$ satisfies
  $P[A^{\delta}]\le 2\delta$.
  Let $\tau^{\delta}$ denote the first time when $Y_{t,x,1}^{\nu^{\delta}}$ reaches the level $\eps_{\delta}$ and set
  $$
  \tilde \nu^{\delta}:=\nu^{\delta}\1_{[t,\tau^{\delta}]}+\1_{(\tau^{\delta},T]} \hat u(\hat X^{\delta}),
  $$
  where $\hat X^{\delta}$ is the solution of~\eqref{eq: SDE hat X} on $[\tau^{\delta},T]$ with initial condition given by $\hat X^{\delta}(\tau^{\delta})=X_{t,x}^{\nu^{\delta}}(\tau^{\delta})$.
  Since the paths of $Y_{t,x,1}^{\nu^{\delta}}$ are nonincreasing, we have
  \[
    \lim\limits_{\delta \downarrow 0}P[X_{t,x}^{\tilde \nu^{\delta}}(T)\ne X_{t,x}^{  \nu^{\delta}}(T) ]
    \le
    \lim\limits_{\delta \downarrow 0}P[\tau^{\delta}\le T ]
    =
    \lim\limits_{\delta \downarrow 0}P[A^{\delta}]=0.
  \]
  Next, we check that $\{f(X^{\nu}_{t,x}(T)),\,\nu \in \cU\}$ is uniformly integrable. This is trivial if $f$ is bounded. Otherwise, Assumption~\ref{ass:D} yields that the coefficients $\mu(x,u)$ and $\sigma(x,u)$ have uniformly linear growth in $x$, and of course they are uniformly Lipschitz in $x$ as they are jointly Lipschitz. Thus $\{X^{\nu}_{t,x}(T),\,\nu \in \cU\}$ is bounded in $L^p$ for any finite $p$ and the uniform integrability follows from the quadratic
  growth assumption~\eqref{eq: cond growth f and g} on $f$. It remains to apply
  Lemma~\ref{lem: cond for right continuity in m}.
\end{proof}

\begin{lemma}\label{lem: local admissibility bar V}
  Let Assumption~\ref{ass:C} hold true and let $B$ be an open neighborhood of $(t,x)\in [0,T]\x \cO$. For all $\nu\in \cU_{t}$ there exists $\bar \nu\in \cU_{t}$ such that
  $$
    \nu\1_{[t,\tau]}+ \bar \nu\1_{(\tau,T]} \;\in\; \bar \cU(t,x),
  $$
  where $\tau$ is the first exit time of $(s,X_{t,x}^{\nu}(s))_{s\ge t}$ from $B$.
\end{lemma}

\begin{proof}
  Let $\hat X_{\tau,X_{t,x}^{\nu}(\tau)}$ be the solution of \eqref{eq: SDE hat X} on $[\tau,T]$ with the (square integrable) initial condition $X_{t,x}^{\nu}(\tau)$ at time $\tau$.
  Then the claim holds true for $\bar \nu:=\nu\1_{[t,\tau]}+\1_{(\tau,T]}\hat u(\hat X_{\tau,X_{t,x}^{\nu}(\tau)})$.
\end{proof}

We have the following counterpart of Lemma~\ref{lem: DPP constraint in expectation}.

\begin{lemma}\label{lem: DPP state constraint}
  (i) Let $B\subseteq [0,T]\times\R^d$ be an open neighborhood of a point $(t,x)\in [0,T]\times\cO$ such that $\bar{V}(t,x)$ is finite and let $\vp: \overline{B} \to \R$ be a continuous function such that $\bar{V}\le \vp$ on $\overline{B}$. For all $\eps>0$ there exists $\nu\in \bar{\cU}(t,x)$ such that
  \[
    \bar V(t,x)\le E \big[ \varphi(\tau,X_{t,x}^\nu(\tau)) \big]+\eps,
  \]
  where $\tau$ is the first exit time of  $(s,X_{t,x}^{\nu}(s))_{s\ge t}$ from $B$.

  (ii) Let Assumptions \ref{ass:C} and \ref{ass:D} hold true and let $B\subseteq [0,T]\times\cO$
  be an open neighborhood of $(t,x)$. For any $\nu\in \cU_{t}$ and any continuous function $\varphi: \overline{B} \to \R$ satisfying $\bar V\ge  \varphi$ on $\overline{B}$,
  \[
    \bar V(t,x) \geq E \big[ \varphi(\tau,X_{t,x}^\nu(\tau)) \big],
  \]
  where $\tau$ is the first exit time of  $(s,X_{t,x}^{\nu}(s))_{s\ge t}$ from $B$.
\end{lemma}

\begin{proof}
  In view of Lemma~\ref{lem: verification assumption section dynamics}, part~(i) is immediate from Theorem~\ref{thm: DPP state constraint}(i). Part~(ii) follows from Theorem~\ref{thm: DPP state constraint}(ii) via Lemmata~\ref{lem: V is right continuous in m in the SDE model} and~\ref{lem: local admissibility bar V} .
\end{proof}

\begin{proof}[Proof of Theorem~\ref{thm: PDE state constraint}]
  The result follows from Lemma~\ref{lem: DPP state constraint} by the arguments in the proof of
  Theorem~\ref{thm: PDE characterization V}.
\end{proof}

\subsubsection{Boundary Condition and Uniqueness}

In this section, we discuss the boundary condition and the uniqueness for the PDE in Theorem~\ref{thm: PDE state constraint}. We shall work under a slightly stronger condition on our setup.

\addtocounter{assumptionBis}{1}
\begin{assumptionBis}\label{ass: mu sigma f for comparison}
  The coefficients $\mu(x,u)$ and $\sigma(x,u)$ in the SDE~\eqref{eq: def X eds} have linear growth in $x$, uniformly in $u$.
\end{assumptionBis}

We also introduce the following regularity condition, which will be used to prove the comparison theorem.

\begin{definition}\label{def: class R}
  Consider a set $\cO\subseteq \R^d$ and a function $w:[0,T]\times \overline{\cO}\to\R$. Then $w$ is \emph{of class $\cR(\cO)$} if the following hold for any $(t,x)\in [0,T)\x\partial \cO$:
  \begin{enumerate}[topsep=3pt, partopsep=0pt, itemsep=1pt,parsep=2pt]
  \item There exist $r>0$, an open neighborhood $B$ of $x$ in $\R^d$ and a function $\ell:\R_{+}\to \R^{d}$ such that
      \begin{align}
         &\liminf_{\eps\to 0 } \eps^{-1} | \ell(\eps)|<\infty \quad \mbox{and } \label{eq: ass delta lambda 2}\\
         &y+ \ell(\eps)+o(\eps)\in \cO \quad\mbox{for all } y \in \overline{\cO}\cap B \mbox{ and } \eps \in (0,r).
         \label{eq: ass delta lambda 1}
       \end{align}
  \item There exists a function $\lambda: \R_+\to \R_{+}$ such that
     \begin{align}
         &\lim_{\eps\to 0} \lambda(\eps)=0 \quad \mbox{and} \label{eq: ass delta lambda 3}\\
         &\lim_{\eps\to 0}w\big(t+\lambda(\eps),x+ \ell(\eps)\big)=w(t,x).\hspace{90pt}\label{eq: ass w2 lambda and delta}
     \end{align}
  \end{enumerate}
\end{definition}

By~\eqref{eq: ass delta lambda 1} we mean that if $\psi:\R_+\to\R^d$ is any function of class $o(\eps)$, then there exists $r>0$ such that $y+ \ell(\eps)+\psi(\eps)\in \cO$ for all $\eps\in(0,r)$. Note that~(i) is a condition on the boundary of $\partial\cO$; it can be seen as a variant of the interior cone condition where the cone is replaced by a more general shape. Condition~(ii) essentially states that $w$ is continuous along at least one curve approaching the boundary point through this shape. In Proposition~\ref{pr:SuffCondForRO} below, we indicate a sufficient condition for $\bar V_{*}$ to be of class $\cR(\cO)$, which is stated directly in terms of the given primitives. We shall see in its proof that Definition~\ref{def: class R} is well adapted to the problem at hand (see also Remark~\ref{rem : comparison assumptions comparison} below).
Before that, let us state the uniqueness result.

\begin{proposition}\label{prop: terminal cond and uniqueness state constraint}
  Let $f$ be continuous and let Assumptions \ref{ass:C} and \ref{ass: mu sigma f for comparison} hold true. Then $\bar V$ has quadratic growth and
  the boundary condition is attained in the sense that
  \[
    \bar V^{*}(T,\cdot)\le f\quad\mbox{and}\quad \bar V_{*}(T,\cdot)\ge f \quad\mbox{on}\quad \overline{\cO}.
  \]
  Assume in addition that $\bar V_{*}$ is of class $\cR(\cO)$. Then
  \begin{enumerate}[topsep=3pt, partopsep=0pt, itemsep=1pt,parsep=2pt]
  \item $\bar V$ is continuous on  $[0,T]\x  \cO$ and admits  a continuous extension to  $[0,T]\x \overline{\cO}$,
  \item $\bar V$ is the  unique (discontinuous) viscosity solution of the state constraint problem
    \[
      -\partial_{t} \vp+\bar H(\cdot,D\vp,D^{2}\vp)=0,\quad \vp(T,\cdot)=f
    \]
    in the class of functions having polynomial growth and having a lower semicontinuous envelope of class $\cR(\cO)$.
  \end{enumerate}
\end{proposition}

\begin{proof}
  Recalling that $f$ has quadratic growth~\eqref{eq: cond growth f and g}, it follows from standard estimates for the SDE~\eqref{eq: def X eds} under Assumptions \ref{ass:C} and \ref{ass: mu sigma f for comparison} that $\bar V$ has quadratic growth and satisfies the boundary condition.

  Assumption~\ref{ass: mu sigma f for comparison} implies Assumption~\ref{ass:D} and hence Theorem~\ref{thm: PDE state constraint} yields that $\bar V^{*}$ and $\bar V_{*}$ are sub- and supersolutions, respectively. Moreover, Assumption~\ref{ass: mu sigma f for comparison} implies that $\bar H$ is continuous and satisfies Assumption~\ref{ass: regu bold H} in Appendix~\ref{se:appendixComparison}; see Lemma~\ref{lem: verification H for the example} below for the proof. If $\bar V_{*}$ is of class $\cR(\cO)$, the comparison principle (Theorem~\ref{thm: comparison} and the subsequent remark, cf.\ Appendix~\ref{se:appendixComparison}) yields that $\bar V^{*}=\bar V_{*}$ on  $[0,T]\x \overline{\cO}$; in particular, (i) holds. Part~(ii) also follows from the comparison result.
\end{proof}

We conclude this section with a sufficient condition ensuring that $\bar V_{*}$ is of class $\cR(\cO)$; the idea is that the volatility should degenerate so that the state process can be pushed away from the boundary. We remark that conditions in a similar spirit exist in the previous literature (e.g., \cite{Ka94, IsLo02}); cf.\ Remark~\ref{rem : comparison assumptions comparison} below.

\begin{proposition}\label{pr:SuffCondForRO}
  Assume that $\bar V_{*}$ is finite-valued on $[0,T]\times\overline{\cO}$ and that $\cO$, $\mu$ and $\sigma$ satisfy the following conditions:
  \begin{enumerate}[topsep=3pt, partopsep=0pt, itemsep=1pt,parsep=2pt]
    \item There exists a $C^{1}$-function $\delta$, defined on  a neighborhood of $\overline\cO\subseteq\R^d$, such that
        $D\delta$ is locally Lipschitz continuous and
        \[
          \delta>0\mbox{ on }\cO,\quad \delta=0\mbox{ on }\partial \cO, \quad \delta<0\mbox{ outside }\overline\cO.
        \]
    \item There exists a locally Lipschitz continuous mapping $\check u: \R^{d}\to U$ such that for all $x\in\overline{\cO}$ there exist an open neighborhood $B$ of $x$ and $\iota>0$ satisfying
       \begin{equation}\label{eq: hat u strictly inside}
       \mu(z,\check u(z))^{\top} D\delta(y) \ge \iota \,\mbox{ and }\,   \sigma(y,\check u(y))=0 \,\mbox{ for all } y  \in B\cap \overline{\cO} \mbox{ and } z\in B.
       \end{equation}
  \end{enumerate}
  Then $\bar V_{*}$ is of class $\cR(\cO)$.
\end{proposition}

\begin{proof}
  Fix $(t,x)\in [0,T)\x \partial \cO$ and let $\delta$, $\iota$ and $B$ and  be as above; we may assume that $B$ is bounded.
  Consider $y\in\overline{\cO}\cap B$. Since $x'\mapsto \mu(x',\check u(x'))$ is locally Lipschitz,
  the ordinary differential equation
  \[
    \check x(s)=y+\int_{0}^{s} \mu\big(\check x(r),\check u(\check x(r))\big)\,dr
  \]
  has a unique solution $\check x_{y}$ on some interval $[0,T_y)$.
  Then~\eqref{eq: hat u strictly inside} ensures that
  \begin{equation}\label{eq: hat x y in Oc}
    \check x_{y}(s)\in \cO \quad\mbox{for } s\in(0,\eps],\mbox{ for $\eps>0$ small enough,}\mbox{ for all }  y\in  \overline{\cO}\cap B.
  \end{equation}
  For $\eps\in [0,T_x)$, we set
  \[
    \ell( \eps):=\check x_{x}(\eps)-x,\quad \lambda(\eps):=\eps.
  \]
  Then~\eqref{eq: ass delta lambda 3} is clearly satisfied. Using the continuity of $\mu$ and $\check u$, we see that
  \[
    \eps^{-1} | \ell( \eps)|\to \mu(x,\check u(x)),
  \]
  which implies \eqref{eq: ass delta lambda 2}.
  Moreover, using that $\mu$ and $\check u$ are locally Lipschitz, we find that
  \[
    \check x_{x}(r)-\check x_{y}(r)=x-y+O(r).
  \]
  Together with~\eqref{eq: hat u strictly inside}, \eqref{eq: hat x y in Oc} and the local Lipschitz continuity of $\mu, \check u$ and $D\delta$, this implies that for $y\in \overline{\cO}$ sufficiently close to $x$ and $\eps>0$ small enough,
  \begin{align*}
    &\delta\big(y+\ell( \eps)+o(\eps)\big) \\[.1em]
    &=\delta\big(y-x+\check x_{x}(\eps)\big) + o(\eps) \\
    &= \delta(y)+\int_{0}^{\eps} \! \mu\big(\check x_{x}(r),\check u(\check x_{x}(r))\big)^{\!\top} D\delta\big(y-x+\check x_{x}(r)\big)\,dr + o(\eps) \\
    &= \delta(y)+ \!\int_{0}^{\eps} \! \mu\big(x-y+\check x_{y}(r),\check u(x-y+\check x_{y}(r))\big)^{\!\top} D\delta\big( \check x_{y}(r)\big)\,dr+O(\eps^{2}) + o(\eps)  \\
    &\ge \eps\iota  + o(\eps),
  \end{align*}
  which is strictly positive for $\eps>0$ small enough. This implies~\eqref{eq: ass delta lambda 1}.

  Consider $(s,y)\in[0,T)\x\cO$ close to $(t,x)$. For $\eps>0$ small enough,  we can find $\nu^{\eps}\in \bar \cU(s+\lambda(\eps),\check x_{y}(\eps))$ such that
  \[
    E\big[f\big(X_{s+\lambda(\eps),\check x_{y}(\eps)}^{\nu^{\eps}}(T)\big)\big]\ge \bar V\big(s+\lambda(\eps),\check x_{y}(\eps)\big)-\eps.
  \]
  Recall the degeneracy condition in~\eqref{eq: hat u strictly inside}. Setting
  \[
    \bar \nu^{\eps}:=\check u(\check x_{y})\1_{[s,s+\lambda(\eps)]}+ \1_{(s+\lambda(\eps),T]}\nu^{\eps},
  \]
  we obtain that
  \begin{align*}
  \bar V(s,y)
    & \geq E\big[f\big(X_{s,y}^{\bar \nu^{\eps}}(T)\big)\big] \\
    & = E\big[f\big(X_{s+\lambda(\eps),\check x_{y}(\eps)}^{\nu^{\eps}}(T)\big)\big] \\
    & \geq \bar V\big(s+\lambda(\eps),\check x_{y}(\eps)\big)-\eps.
  \end{align*}
  Recalling that $\check x_{y}(\eps)\to \check x_{x}(\eps)$ as $y\to x$, this leads to
  \[
    \bar V_{*}(t,x)\ge \bar V_{*}\big(t+\lambda(\eps),\check x_{x}(\eps)\big)-\eps=\bar V_{*}\big(t+\lambda(\eps),x+\ell(\eps)\big)-\eps
  \]
  for $\eps>0$ small enough, which implies
  \begin{align*}
    \bar V_{*}(t,x)
      &\ge \limsup_{\eps\to 0}\bar V_{*}\big(t+\lambda(\eps),x +\ell(\eps)\big)\\
      &\ge \liminf_{\eps\to 0}\bar V_{*}\big(t+\lambda(\eps),x +\ell(\eps)\big)\\
      & \ge \bar V_{*}(t,x).
  \end{align*}
  Hence $\lim_{\eps\to 0}\bar V_{*}\big(t+\lambda(\eps),x +\ell(\eps)\big)=\bar V_{*}(t,x)$; i.e., \eqref{eq: ass w2 lambda and delta} holds for $\bar V_{*}$.
\end{proof}

\subsubsection{On Closed State Constraints}

Recall that the value function $\bar V$ considered above corresponds to the constraint that the state processes remains in the open set $\cO$.
We can similarly consider the closed constraint; i.e.,
\[
  \overline{V}(t,x):=\sup \big\{E[f(X_{t,x}^{\nu}(T))]:\,\nu\in\cU_t,\, X_{t,x}^\nu(s)\in \overline{\cO}\mbox{ for all }s\in [t,T],\,P\as\big\}
\]
The arguments used above for $\bar V$ do not apply to $\overline{V}$. Indeed, for the closed set, the constraint function $G$ in Section~\ref{sec: DPP state constraint} would not be u.s.c.\ and hence the derivation of Theorem~\ref{thm: DPP state constraint} fails; note that the upper semicontinuity is essential for the covering argument in the proof of Theorem~\ref{th:DPP}(ii),(ii'). Moreover, the switching argument in the proof of Lemma~\ref{lem: local admissibility bar V} cannot be imitated since, given that the state process $X_{t,x}^\nu$ hits the boundary $\partial \cO$, it is not possible to know which trajectories of the state will actually exit $\overline{\cO}$.

However, we shall see that, if a comparison principle holds, then the dynamic programming principle for the open constraint $\cO$ is enough to fully characterize the value function $\overline{V}$ associated to $\overline{\cO}$. More precisely, we shall apply the PDE for $\bar V$ and its comparison principle to deduce that $\overline{V}=\bar V$ under certain conditions.
Of course, the basic observation that this equality holds under suitable conditions is not new; see, e.g., \cite{IsLo02}.
We shall use the following assumption.

\addtocounter{assumptionBis}{-2}
\begin{assumptionBis}\label{ass:C'}
  Assumption~\ref{ass:C} holds with $\hat{u}$ defined on $\overline{\cO}$.
\end{assumptionBis}

\begin{corollary}\label{co:closedConstraint}
  Let $f$ be continuous, let Assumptions \ref{ass:C'} and \ref{ass: mu sigma f for comparison} hold true and assume that
  $\bar V_{*}$ is of class $\cR(\cO)$. Then $\overline{V}=\bar V$ on $[0,T]\times \overline{\cO}$.
\end{corollary}

We recall that $\bar V$ admits a (unique) continuous extension to  $[0,T]\x \overline{\cO}$ under the stated conditions, so the assertion makes sense.

\begin{proof}
  The easier direction of the dynamic programming principle for $\overline{V}$ can be obtained as above, so the result of
  Lemma~\ref{lem: DPP state constraint}(i) still holds. It then follows by the same arguments as in the proof of
  Theorem~\ref{thm: PDE state constraint} and Proposition \ref{prop: terminal cond and uniqueness state constraint} that the upper semicontinuous envelope $\overline{V}^{*}$ is a viscosity subsolution of \eqref{eq: PDE subsol thm state constraint} satisfying $\overline{V}^{*}(T,\cdot)\le f$ on $\overline \cO$. As in the proof of Proposition~\ref{prop: terminal cond and uniqueness state constraint}, we can apply the comparison
  principle of Theorem~\ref{thm: comparison} to deduce that $\overline{V}^{*}\leq \bar V_*$ on $[0,T]\x \overline{\cO}$.
  On the other hand, we clearly have $\bar V \leq \overline{V}$ on $[0,T]\x \cO$ by the definitions of these value functions. Therefore, we have
  \[
   \bar V_{*}\le \overline{V}_* \leq \overline{V}^{*}\leq \bar V_* =\bar V^{*},
  \]
  where $\overline{V}_{*}$ denotes the lower semicontinuous envelope of $\overline{V}$ and the last equality is due to  Proposition~\ref{prop: terminal cond and uniqueness state constraint}.
  It follows that all these functions coincide.
\end{proof}

\appendix
\section{Comparison for State Constraint Problems}\label{se:appendixComparison}

In this appendix we provide, by adapting the usual techniques, a fairly general comparison theorem for state constraint problems which is suitable for the applications in Proposition~\ref{prop: terminal cond and uniqueness state constraint} and Corollary~\ref{co:closedConstraint}.

In the following, $\cH$ denotes  a continuous mapping from $\R^{d}\x \R^{d}\x\M^{d}$ to $ \R$ which is nonincreasing in its third variable, $\cO$ is a given open subset of $\R^{d}$, and $\rho>0$ is a fixed constant.
We consider the equation
\begin{equation}\label{eq: PDE H}
  \rho \vp -\partial_{t} \vp + \cH(\cdot,D\vp,D^{2}\vp)= 0
\end{equation}
and the following condition on $\cH$.

\begin{assumption}\label{ass: regu bold H}
  There exists $\alpha>0$  such that
  \begin{align*}
    \liminf_{\eta\downarrow 0} \big(&\cH(y,q,Y^{\eta})-\cH(x,p,X^{\eta})\big)\\
    &\leq \alpha \Big( |x-y|\big(1+|q| + n^{2} |x-y|\big) + (1+|x|)|p-q|+(1+|x|^{2})|Q| \Big)
  \end{align*}
  for all $(x,y)\in \overline{\cO}$ with $|x-y|\le 1$ and for all $(p,q,Q)\in \R^{d}\x\R^{d}\x \M^{2d}$, $(X^{\eta},Y^{\eta})_{\eta>0}\subset \M^{d}\x\M^{d}$ and $n\ge 1$ such that
  \[
    \left(\begin{array}{cc} X^{\eta}& 0 \\ 0 & -Y^{\eta}\end{array}\right)
    \le
    A_{n}+\eta A_{n}^{2}\quad\mbox{for all } \eta>0,
  \]
  where
  \[
  A_{n}:=n^{2}\left(\begin{array}{cc} I_{d}& -I_{d} \\ -I_{d} & I_{d}\end{array}\right)+Q.
  \]
\end{assumption}

\begin{theorem}\label{thm: comparison}
  Let Assumption~\ref{ass: regu bold H} hold true. Let $w_{1}$ be an u.s.c.\ viscosity subsolution on $\overline{\cO}$ and let $w_2$ be an l.s.c.\ viscosity supersolution on  $\cO$ of \eqref{eq: PDE H}. If $w_1$ and $w_2$ have polynomial growth  on $\overline{\cO}$ and if $w_{2}$ is of class $\cR(\cO)$, then
  \[
   w_{2}\ge w_{1}\;\;\mbox{on }\{T\}\x \overline{\cO}\;\quad\mbox{implies}\;\quad w_{2}\ge w_{1}\;\;\mbox{on }[0,T]\x \overline{\cO}.
  \]
\end{theorem}

\begin{remark}{\rm
  Our result also applies to the equation
  \begin{equation}\label{eq:dtVersion}
    -\partial_{t} \vp + \cH(\cdot,D\vp,D^{2}\vp)= 0,
  \end{equation}
  provided that $\cH$ is homogeneous of degree one with respect to  its second and third argument, as it is the case for the Hamilton-Jacobi-Bellman operators in the body of this paper. Indeed, $w_{1}$ is then a subsolution of~\eqref{eq:dtVersion} if and only if  $(t,x)\mapsto e^{\rho t }w_{1}(t,x)$ is a subsolution of~\eqref{eq: PDE H}, and similarly for the supersolution. Further extensions could also be considered but are beyond the scope of this paper.
  }
\end{remark}

\begin{proof}[Proof of Theorem~\ref{thm: comparison}]
  Assume that $w_{2}\ge w_{1}$ on $\{T\}\x \overline{\cO}$. Let $p\ge 1$ and $C>0$ be such that $w_{1}(t,x)-w_{2}(t,x)\le C(1+|x|^{p})$ for all $(t,x)\in [0,T]\x \overline{\cO}$. Assume for contradiction that $\sup (w_{1}-w_{2})>0$, then we can find $\iota>0$ and $(t_{0},x_{0})\in  [0,T]\x \overline{\cO}$ such that
  \begin{equation}\label{eq: proof comp cara point eta}
  \zeta:=(w_{1}-2\phi-w_{2})(t_{0},x_{0})
  =\max_{[0,T]\x \overline{\cO} } (w_{1}-2\phi-w_{2})>0,
  \end{equation}
  where
  \[
    \phi(t,x):=\iota e^{-\kappa t}(1+|x|^{2p}).
  \]
  Here $\kappa>0$ is a fixed constant which is large enough to ensure that
  \begin{align}\label{eq:mFunctionDef}
    m(t&,x):= \\
   & -\rho\phi(t,x)+\partial_{t}\phi(t,x)
   +\alpha\big((1+|x|) |D\phi(t,x)| +(1+|x|^{2})|D^{2}\phi(t,x)|\big) \nonumber
  \end{align}
  is nonpositive on $[0,T]\x \R^{d}$, where $\alpha$ is the constant from Assumption~\ref{ass: regu bold H}.

  Note that by the assumption that $w_{2}\ge w_{1}$ on $\{T\}\x \overline{\cO}$, we must have $(t_{0},x_{0})\in [0,T)\x \overline{\cO}$.

  \emph{Case 1:  $x_{0} \in \partial \cO$.}
  For all $n\ge 1$, there exist $(t^{n},x^{n},s^n,y^n)\in( [0,T]\x \overline{\cO})^{2} $ satisfying
  \begin{equation}\label{eq: proof comp cara point eta n}
  \Phi^{n}(t^{n},x^{n},s^{n},y^{n})=\max\limits_{( [0,T]\x \overline{\cO})^{2}}\Phi^{n},
  \end{equation}
  where
  \[
  \Phi^{n}(t,x,s,y):=w_{1}(t,x)-w_{2}(s,y)-\Theta^{n}(t,x,s,y)
  \]
  and
  \begin{align*}
    \Theta^{n}(t,x,s,y)
      &:= \frac12 n^{2}\big(|t+\lambda(n^{-1})-s|^{2}+\eps|x+ \ell(n^{-1})-y|^{2}\big) \\
      & \phantom{:=\;} + |t-t_{0}|^{2}+|x-x_{0}|^{4}+\phi(t,x)+\phi(s,y),
  \end{align*}
  with $\ell$ and $\lambda$ given for $x_{0}$ as in the statement of the Definition \ref{def: class R}, and $\eps>0$.
  Note that \eqref{eq: proof comp cara point eta n}, the assumption that $w_2$ satisfies \eqref{eq: ass w2 lambda and delta}, and \eqref{eq: proof comp cara point eta} imply that
  \begin{align}\label{eq:proofComparisonZeta}
    \Phi^{n}(t^{n},x^{n},s^{n},y^{n})
    &\ge \Phi^{n}\big(t_{0},x_{0},t_{0}+\lambda(n^{-1}),x_{0}+ \ell(n^{-1})\big) \nonumber\\
    &=  (w_{1}-2\phi-w_{2})(t_{0},x_{0})+o(1) \nonumber\\
    &= \zeta+o(1)\quad \mbox{as }n\to\infty.
  \end{align}
  Recalling the growth condition on $w_{1}, w_{2}$ and the definition of $\phi$, it follows that, after passing to a subsequence, $(t^{n},x^{n},s^n,y^n)$ converges to some $(t^{\infty},x^{\infty},t^{\infty},x^{\infty})\in ([0,T]\x\overline{\cO})^{2}$.
  We then have
  \begin{align*}
    \zeta
     & = (w_{1}-2\phi-w_{2})(t_{0},x_{0}) \\
     & =\max_{[0,T]\x \overline{\cO} } (w_{1}-2\phi-w_{2})\\
     &\ge (w_{1}-2\phi-w_{2})(t^\infty,x^{\infty})- |t^{\infty}-t_{0}|^{2}-|x^{\infty}-x_{0}|^{4} \\
     &\phantom{\ge\;} - \limsup_{n\to \infty}  \frac12 n^{2}\big(|t^{n}+\lambda(n^{-1})-s^{n}|^{2}+\eps|x^{n}+ \ell(n^{-1})-y^{n}|^{2}\big)\\
     & \geq \liminf_{n\to\infty} \Phi^{n}(t^{n},x^{n},s^{n},y^{n}) \\
     &\ge  \zeta,
  \end{align*}
  where~\eqref{eq:proofComparisonZeta} was used in the last step. After passing to a subsequence, we deduce that
  \begin{eqnarray}
  &(t^{n},x^{n},s^{n},y^{n})\to (t_{0},x_{0},t_{0},x_{0}),& \label{eq: proof comp conv point eta n}\\
  &w_{1}(t^{n},x^{n})-w_{2}(s^{n},y^{n})\to  (w_{1}-w_{2})(t_{0},x_{0}),& \label{eq: proof comp conv point eta n through functions}
  \\
  &s^{n}=t^{n}+\lambda(n^{-1})+o(n^{-1}),\quad  y^{n}=x^{n}+\ell(n^{-1})+o(n^{-1}) .& \label{eq: proof comp conv point eta n with speed}
  \end{eqnarray}
  Since $(t_{0},x_{0}) \in [0,T)\x\partial \cO$, it follows from \eqref{eq: ass delta lambda 1}, \eqref{eq: ass delta lambda 3} and \eqref{eq: proof comp conv point eta n with speed} that $(s^{n},y^{n})\in [0,T)\x \cO$ for $n$ large enough.

  Let $\overline \cP^{2,+}_{\cO}w_{1}$ and $ \overline \cP^{2,-}_{\cO}w_{2}$ be the
  ``closed'' parabolic super- and subjets as defined in \cite[Section~8]{CrIsLi92}.
  From the Crandall-Ishii lemma \cite[Theorem 8.3]{CrIsLi92} we obtain, for each $\eta>0$, elements
  \[
    (a^{n},p^{n},X^{n}_{\eta}) \in \overline \cP^{2,+}_{\cO}w_{1}(t^{n},x^{n})\quad \;\mbox{and}\;\quad (b^{n},q^{n},Y^{n}_{\eta}) \in \overline \cP^{2,-}_{\cO}w_{2}(s^{n},y^{n})
  \]
  such that
  \begin{align*}
    a^{n}&=\partial_{t} \Theta^{n}(t^{n},x^{n},s^{n},y^{n}),  && b^{n}=-\partial_{s} \Theta^{n}(t^{n},x^{n},s^{n},y^{n}),\\
    p^{n}&=D_{x}\Theta^{n}(t^{n},x^{n},s^{n},y^{n}),  && q^{n}=-D_{y} \Theta^{n}(t^{n},x^{n},s^{n},y^{n}),
  \end{align*}
  \[
   \left(\begin{array}{cc} X^{n}_{\eta}& 0 \\ 0 & -Y^{n}_{\eta}\end{array}\right)
   \le
  A_{n}+\eta A_{n}^{2},
  \]
  where $A_n=D^2 \Theta^n(t^{n},x^{n},s^{n},y^{n})$; i.e.,
  \[
    A_{n}=\eps n^{2}\left(\begin{array}{cc} I_{d}& -I_{d} \\ -I_{d} & I_{d}\end{array}\right)+
    \left(\begin{array}{cc} D^{2}\phi(t^{n},x^{n})  +O(|x^{n}-x_{0}|^2)& 0 \\ 0 & D^{2}\phi(s^{n},y^{n}) \end{array}\right).
  \]
  In view of the sub- and supersolution properties of $w_{1}$ and $w_{2}$, the fact that $(s^{n},y^{n})\in [0,T)\x \cO$ for $n$ large, and Assumption~\ref{ass: regu bold H}, we deduce that
  \begin{align*}
   \Delta_{n}
   &:=\rho(w_{1}(t^{n},x^{n})-w_{2}(s^{n},y^{n}))\\
   &\le 2 (t^{n}-t_{0})+\partial_{t}\phi(t^{n},x^{n})+\partial_{s}\phi(s^{n},y^{n})\\
   &\phantom{\le\,}+ \alpha   |x^{n}-y^{n}|\big(1+|q^{n}| + \eps n^{2} |x^{n}-y^{n}|\big)\\
   &\phantom{\le\,} +\alpha\big( (1+|x^{n}|)|q^{n}-p^{n}|+(1+|x^{n}|^{2})|Q^{n}| \big),
  \end{align*}
  where
  \[
   Q^{n}:= \left(\begin{array}{cc} D^{2}\phi(t^{n},x^{n})  +O(|x^{n}-x_{0}|^2)& 0 \\ 0 & D^{2}\phi(s^{n},y^{n}) \end{array}\right).
  \]
  By the definitions of $p^{n}$ and $q^{n}$, it follows that
  \begin{align*}
   \Delta_{n}
   &\le 2 (t^{n}-t_{0})+\partial_{t}\phi(t^{n},x^{n})+\partial_{s}\phi(s^{n},y^{n})\\
   &\phantom{\le\,} + \alpha|x^{n}-y^{n}|\big(1  +|D\phi(s^{n},y^{n})|+\eps n^{2}|x^{n}+\ell(n^{-1})-y^{n}|+  \eps n^{2} |x^{n}-y^{n}|\big)\\
   &\phantom{\le\,}+\alpha  (1+|x^{n}|)\big(4|x^{n}-x_{0} |^{3} + |D\phi(t^{n},x^{n})|+ |D\phi(s^{n},y^{n})|\big)\\
   &\phantom{\le\,}+ \alpha(1+|x^{n}|^{2})\big( |D^{2}\phi(t^{n},x^{n})|+ |D^{2}\phi(s^{n},y^{n})|+O(|x^{n}-x_{0}|^2)\big).
  \end{align*}
  Recalling \eqref{eq: proof comp conv point eta n}--\eqref{eq: proof comp conv point eta n with speed}, letting $n\to\infty$ leads to
   \begin{align*}
   \rho(w_{1} -w_{2} )(t_{0},x_{0})
   &\le 2\partial_{t} \phi(t_{0},x_{0})  + \alpha \eps \Big(\liminf_{n\to \infty}n\ell(n^{-1})\Big)^2\\
   &\phantom{\le\;} +2\alpha\big((1+|x_{0}|)|D\phi(t_{0},x_{0})|+(1+|x_{0}|^{2})|D^{2}\phi(t_{0},x_{0})|\big),
  \end{align*}
  which, by  \eqref{eq: ass delta lambda 2} and the definition of $m$ in \eqref{eq:mFunctionDef}, implies
  \[
   \rho(w_{1} -2\phi -w_{2})(t_{0},x_{0}) \le 2m(t_{0},x_{0})
  \]
  after letting $\eps \to 0$. Since $\kappa>0$ has been chosen so that $m\le 0$ on $[0,T]\x \R^{d}$, this contradicts~\eqref{eq: proof comp cara point eta}.

  \emph{Case 2:  $x_{0} \in \cO$.} This case is handled similarly by using
  \[
   \Theta^{n}(t,x,s,y):=\frac12 n^{2}\big(|t-s|^{2}+|x-y|^{2}\big)
   +|t-t_{0}|^{2}+|x-x_{0}|^{4}+\phi(t,x)+\phi(s,y).
  \]
  After taking a subsequence, the corresponding sequence of maximum points $(t^{n},x^{n},s^{n},y^{n})_{n\ge 1}$ again converges to $(t_{0},x_{0},t_{0},x_{0})$, so that $x^{n},y^{n}\in \cO$ for $n$ large enough. The rest of the proof follows the same arguments as in Case~1.
\end{proof}

\begin{lemma}\label{lem: verification H for the example}
  Under Assumption~\ref{ass: mu sigma f for comparison}, the operator $\bar H$ defined in~\eqref{eq: def bar H example} satisfies Assumption~\ref{ass: regu bold H}.
\end{lemma}

\begin{proof}
  The argument is standard. We first observe that
  \begin{align*}
    -\bar L^{u}(y,q,Y^{\eta})+\bar L^{u}(x,p,X^{\eta})
    &=\big(\mu(x,u)-\mu(y,u) \big)^{\top}q + \mu(x,u)^{\top} (p-q)\\
    &\phantom{=}\;+ \frac12 \sum_{1\le i\le d}\Sigma^{i}(x,y,u)^\top\,\Xi^{\eta}\;\Sigma^{i}(x,y,u),
  \end{align*}
  where
  \[
    \Sigma(x,y,u):=\left(\begin{array}{cc}\sigma(x,u)\\\sigma(y,u)\end{array}\right)\quad\mbox{and}\quad \Xi^{\eta}:= \left(\begin{array}{cc} X^{\eta}& 0 \\ 0 & -Y^{\eta}\end{array}\right)
  \]
  and $\Sigma^{i}$ denotes the $i$th column of $\Sigma$.
  Since $\mu$ is Lipschitz continuous and has uniformly linear growth by Assumption~\ref{ass: mu sigma f for comparison}, we have
  \begin{eqnarray*}
  \big(\mu(x,u)-\mu(y,u) \big)^{\top}q
  +
  \mu(x,u)^{\top} (p-q)
  \le \alpha\big( |x-y||q|+(1+|x|)|p-q|\big)
  \end{eqnarray*}
  for some constant $\alpha>0$. Recall from the statement of Assumption~\ref{ass: regu bold H} the condition
  that
  \[
   \Xi^{\eta}\le A_{n}+\eta A_{n}^{2},\quad\mbox{where}\quad
    A_{n}:=n^{2}\left(\begin{array}{cc} I_{d}& -I_{d} \\ -I_{d} & I_{d}\end{array}\right)+Q.
  \]
  Since $\sigma$ is Lipschitz continuous and satisfies Assumption \ref{ass: mu sigma f for comparison}, it follows that
    \begin{align*}
    \Sigma^{i}(x,y,u)^\top \,\Xi^{\eta} \,\Sigma^{i}(x,y,u)
    &\le n^{2}|\sigma^{i}(x,u)-\sigma^{i}(y,u)|^{2}\\
    &\phantom{\le}\; + \alpha\big(1+|x|^{2}+|y|^{2}\big)|Q|+O(\eta),
  \end{align*}
   possibly after enlarging $\alpha$, where $\sigma^{i}$ denotes the $i$th column of $\sigma$. We conclude by using the Lipschitz continuity of $\sigma$ and the condition that $|x-y|\le 1$.
\end{proof}

\begin{remark}\label{rem : comparison assumptions comparison}{\rm
  We conclude with some comments on Theorem~\ref{thm: comparison} and Proposition~\ref{pr:SuffCondForRO}, and related results in the literature.

 The main issue in proving comparison with state constraints is to avoid boundary points; i.e., that $y_{n}$ (see the proof of Theorem~\ref{thm: comparison}) ends up on the boundary. One classical way to ensure this, is to use a perturbation of $|x-y|^{2}$ in a suitable inward direction, like the function $\ell$ above. Moreover, this requires the supersolution to be continuous at the boundary points, along the direction of perturbation; cf.~\eqref{eq: ass w2 lambda and delta}.  %

 In~\cite{Ka94}, the inner normal $n(x)$ at the boundary point $x\in \partial \cO$ is used as an inward direction. In the proof, one is only close to $x$ (cf.\ \eqref{eq: proof comp conv point eta n with speed}); therefore, the comparison result~\cite[Theorem 2.2]{Ka94} requires the existence of a truncated cone around $n(x)$ which stays inside the domain, and the continuity of the subsolution along the directions that it generates. Our condition~\eqref{eq: ass delta lambda 1} is less restrictive than the corresponding requirement in~\cite{Ka94}: we only need the continuity along the curve $\eps\mapsto \ell(\eps)$, cf.~\eqref{eq: ass w2 lambda and delta}, rather than all lines in a neighborhood. The function $\lambda$ appears because we consider parabolic equations, whereas~\cite{Ka94} focuses on the elliptic case.

 In Proposition~\ref{pr:SuffCondForRO} we give conditions (certainly not the most general possible) ensuring that the value function is of class $\cR(\cO)$. They should be compared to~\cite[Condition~(A3)]{Ka94}, which is used to verify the continuity assumption of~\cite[Theorem 2.2]{Ka94}. Our conditions are stronger in the sense that they are imposed around the boundary and not only at the boundary; on the other hand, we require $C^{1}$-regularity of the boundary whereas~\cite{Ka94} requires $C^{3}$.

  In~\cite{IsLo02}, a slightly different technique is used, based on ideas from~\cite{IshiiKoike.96}. First, it is assumed that at each boundary point $x$, there exists a fixed control which kills the volatility at the neighboring  boundary points and keeps a truncated cone around the drift in the domain. This is similar to our~\eqref{eq: hat u strictly inside}, except that our control is not fixed; on the other hand, we assume that it kills the volatility in a neighborhood of $x$. Thus, the conditions in~\cite{IsLo02} are not directly comparable to ours; e.g., if $\cO$ is the unit disk in $\R^{2}$, $U=[-1,1]$, $\mu(x,u)=-x$ and $\sigma(x,u)=|x^{2}-u|I_{2}$, then~\eqref{eq: hat u strictly inside} is satisfied (with $\delta$ given by the Euclidean distance near the boundary and $\check u(x)=x^{2}$), while~\cite[Condition~(2.1)]{IsLo02} is not.
  Second, in~\cite{IsLo02}, the state constraint problem is transformed so as to introduce a Neumann-type boundary condition and construct a suitable test function which, as a function of its first component, turns out to be a uniformly strict supersolution of the Neumann boundary condition. The construction in~\cite{IsLo02} heavily relies on the assumption that the coefficients are bounded; cf.\ the beginning of~\cite[Section~3]{IsLo02} and the proof of~\cite[Theorem~3.1]{IsLo02}.
  }
\end{remark}


\begin{thebibliography}{10}
%
%
%
%
%
%
%
%

\bibitem{BEI10}
B.~Bouchard, R.~Elie, and C.~Imbert.
\newblock Optimal control under stochastic target constraints.
\newblock {\em SIAM J. Control Optim.}, 48(5):3501--3531, 2010.

\bibitem{BET09}
B.~Bouchard, R.~Elie, and N.~Touzi.
\newblock Stochastic target problems with controlled loss.
\newblock {\em SIAM J. Control Optim.}, 48(5):3123--3150, 2009.

\bibitem{BT10}
B.~Bouchard and N.~Touzi.
\newblock Weak dynamic programming principle for viscosity solutions.
\newblock {\em SIAM J. Control Optim.}, 49(3):948--962, 2011.

\bibitem{BV12}
B.~Bouchard and T.N.~Vu.
\newblock A stochastic target approach for P\&L matching problems.
\newblock {\em Mathematics of Operation Research}, 37(3):526--558, 2012.

\bibitem{CrIsLi92}
M.~Crandall, H.~Ishii, and P.-L.~Lions.
\newblock User's guide to viscosity solutions of second order partial
  differential equations.
\newblock {\em Bull. Amer. Math. Soc.}, 27(1):1--67, 1992.

\bibitem{FlemingSoner.06}
W.~H. Fleming and H.~M. Soner.
\newblock {\em Controlled Markov Processes and Viscosity Solutions}.
\newblock Springer, New York, 2nd edition, 2006.


\bibitem{freidlin1985functional}
M.~I.~Freidlin.
\newblock {\em Functional Integration and Partial Differential Equations}.
\newblock Princeton University Press, 1985.


\bibitem{follmer2000efficient}
H.~F{\"o}llmer  and P.~Leukert.
\newblock {Efficient hedging: cost versus shortfall risk}.
\newblock {\em Finance and Stochastics}, 4(2):117--146, 2000.

\bibitem{IshiiKoike.96}
H.~Ishii and S.~Koike.
\newblock A new formulation of state constraint problems for first-order
  {PDEs}.
\newblock {\em SIAM J. Control Optim.}, 34(2):554--571, 1996.

\bibitem{IsLo02}
H.~Ishii and P.~Loreti.
\newblock A class of stochastic optimal control problems with state constraint.
\newblock {\em Indiana Univ. Math. J.}, 51(5):1167--1196, 2002.

%
%
%
%

\bibitem{Ka94}
M.~A.~Katsoulakis.
\newblock Viscosity solutions of second order fully nonlinear elliptic equations with state constraints.
\newblock {\em Indiana Univ. Math. J.}, 43(2):493--519, 1994.

\bibitem{LasryLions.89}
J.-M. Lasry and P.-L. Lions.
\newblock Nonlinear elliptic equations with singular boundary conditions and
  stochastic control with state constraints. I. The model problem.
\newblock {\em Math. Ann.}, 283(4):583--630, 1989.

\bibitem{Soner.86a}
H.~M.~Soner.
\newblock Optimal control with state-space constraint. I.
\newblock {\em SIAM J. Control Optim.}, 24(3):552--561, 1986.

\bibitem{Soner.86b}
H.~M.~Soner.
\newblock Optimal control with state-space constraint. II.
\newblock {\em SIAM J. Control Optim.}, 24(6):1110--1122, 1986.

\bibitem{SonerTouzi.02b}
H.~M.~Soner and N.~Touzi.
\newblock Dynamic programming for stochastic target problems and geometric
  flows.
\newblock {\em J. Eur. Math. Soc. (JEMS)}, 4(3):201--236, 2002.

\bibitem{Soravia.99a}
P.~Soravia.
\newblock Optimality principles and representation formulas for viscosity
  solutions of {H}amilton-{J}acobi equations. {I}. {E}quations of unbounded and
  degenerate control problems without uniqueness.
\newblock {\em Adv. Differential Equations}, 4(2):275--296, 1999.

\bibitem{Soravia.99b}
P.~Soravia.
\newblock Optimality principles and representation formulas for viscosity
  solutions of {H}amilton-{J}acobi equations. {II}. {E}quations of control
  problems with state constraints.
\newblock {\em Differential Integral Equations}, 12(2):275--293, 1999.

\bibitem{YongZhou.99}
J.~Yong and X.~Y. Zhou.
\newblock {\em Stochastic Controls. Hamiltonian Systems and HJB Equations}.
\newblock Springer, New York, 1999.
\end{thebibliography}
\end{document}